\documentclass[12pt]{amsart}

\usepackage{bansil-kitagawa-stability}

\renewcommand{\S}{\mathbb{S}}
\begin{document}

\author{Mohit Bansil}
\address{Department of Mathematics, Michigan State University}
\email{bansilmo@msu.edu}

\author{Jun Kitagawa}
\address{Department of Mathematics, Michigan State University}
\email{kitagawa@math.msu.edu}

\title[]{Quantitative stability in the geometry of semi-discrete optimal transport}
\thanks{JK's research was supported in part by National Science Foundation grants DMS-1700094 and DMS-2000128.}

\begin{abstract}
We show quantitative stability results for the geometric ``cells'' arising in semi-discrete optimal transport problems. We first show stability of the associated Laguerre cells in measure, without any connectedness or regularity assumptions on the source measure. Next we show quantitative invertibility of the map taking dual variables to the measures of Laguerre cells, under a Poincar{\`e}-Wirtinger inequality. Combined with a regularity assumption equivalent to the Ma-Trudinger-Wang conditions of regularity in Monge-Amp{\`e}re, this invertibility leads to stability of Laguerre cells in Hausdorff measure and also stability in the uniform norm of the dual potential functions, all  stability results come with explicit quantitative bounds. Our methods utilize a combination of graph theory, convex geometry, and Monge-Amp{\`e}re regularity theory.
\end{abstract}

\maketitle

\tableofcontents

\section{Introduction}

\subsection{Semi-discrete optimal transport}

Let $X\subset \R^n$, $n\geq 2$ compact and $Y:=\{y_i\}_{i=1}^N \subset \R^n$ a fixed finite set, and fix a Borel measurable \emph{cost function} $c: X\times Y\to \R$. If $\mu$ is an absolutely continuous probability measure supported in $X$ and $\nu$ is a discrete probability measure supported on $Y$, then the \emph{semi-discrete optimal transport problem} is to minimize the functional 
\begin{align}\label{eqn: monge}
	\int_X c(x, T(x)) d\mu
\end{align}
over all Borel measurable mappings $T: X\to Y$ such that $T_\#\mu(E):=\mu(T^{-1}(E))=\nu(E)$ for any measurable $E\subset Y$. This problem has been well-studied in the more general case when $\nu $ may not be a discrete measure, and has deep connections to many mathematical areas, as mentioned throughout \cite{Villani09}.

In this paper we are concerned with quantitative stability of the geometric structures when minimizing \eqref{eqn: monge}, under perturbations of the target measure $\nu$. It is known that under some mild conditions, an optimal map $T$ can be constructed via a $\mu$-a.e. partition of the domain $X$ which is induced by a potential function which maximizes an associated dual problem. The cells in such a partition are known as \emph{Laguerre cells} (see Definition \ref{def: lag cell} below). We will show stability of these cells under perturbations of $\nu$ measured in two different ways: an integral notion, and a uniform notion. As a corollary, we will also obtain stability of the associated dual potential functions in uniform norm; all of these stability results will come with explicit quantitative estimates.

For the remainder of the paper, we fix positive integers $N$ and $n$ and a collection $Y:=\{y_i\}_{i=1}^N\subset \R^n$. We also define
\begin{align*}
	\weightvectset:=\{{\weightvect}\in \R^N\mid \sum_{i=1}^N\weightvect^i=1,\ \weightvect^i\geq 0\},
\end{align*}
and to any vector ${\weightvect}\in \weightvectset$ we associate the discrete measure $\displaystyle\nu_{{\weightvect}}:=\sum_{i=1}^N\weightvect^i\delta_{y_i}$, and we let $\onevect=(1,\ldots, 1)\in\R^N$. Superscripts will be used for coordinates of a vector, and we use $\norm{V}:=\sqrt{\sum_{i=1}^N \abs{V^i}^2}$ for the Euclidean ($\ell^2$) norm of a vector $V\in \R^N$, while $\norm{V}_1:=\sum_{i=1}^N \abs{V^i}$ and $\norm{V}_\infty:=\max_{i\in\{1,\ldots, N\}}\abs{V^i}$ are respectively the $\ell^1$ and $\ell^\infty$ norms. We may also use $\norm{T}$ for the operator norm of a linear transformation $T$, this will be clear from context. Lastly, $\L$ will denote $n$ dimensional Lebesgue measure.
\subsection{Statement of results}
We assume the following standard conditions on $c$ throughout:
\begin{align}
	c(\cdot, y_i)&\in C^2(X), \forall i\in \{1, \ldots, N\},\label{Reg}\tag{Reg}\\
	\nabla_xc(x, y_i)&\neq \nabla_xc(x, y_k),\ \forall x\in X,\ i\neq k.\label{Twist}\tag{Twist}
\end{align}
These two conditions are standard in the existence theory for optimal transport, see \cite{MaTrudingerWang05}. We then make the following definitions:
\begin{defin}\label{def: lag cell}
	If $\varphi: X\to \R\cup \{+\infty\}$  (not identically $+\infty$) and $\psi\in \R^N$, their $c$- and $c^*$-transforms are a vector $\varphi^c\in \R^N$ and a function $\psi^{c^*}: X\to \R\cup \{+\infty\}$ respectively, defined by
	\begin{align*}
		(\varphi^c)^i:=\sup_{x\in X}(-c(x, y_i)-\varphi(x)),\quad (\psi^{c^*})(x):=\max_{i\in \{1, \ldots, N\}}(-c(x, y_i)-\psi^i).
	\end{align*}
	For 
	 $i\in \{1, \ldots, N\}$, the \emph{$i$th Laguerre cell} associated to $\psi$ is defined by
	\begin{align*}
		\Lag_i(\psi):=\{x\in X\mid -c(x, y_i)-\psi^i=\psi^{c^*}\}.
	\end{align*}
	We also define the map $G: \R^N\to \weightvectset$ by
	\begin{align*}
		G(\psi):=(G^1(\psi), \ldots, G^N(\psi))=(\mu(\Lag_1(\psi)), \ldots, \mu(\Lag_N(\psi))),
	\end{align*}
	and define for any $\epsilon\geq 0$,
	\begin{align*}
		\mathcal{K}^\epsilon:=\{\psi\in \R^N\mid G^i(\psi)> \epsilon,\ \forall i\in \{1, \ldots, N\}\}.
	\end{align*}
\end{defin}
	When $\mu$ is absolutely continuous with respect to Lebesgue measure, it is clear that \eqref{Twist} implies Laguerre cells for a fixed $\psi$ associated to different indices are disjoint up $\mu$-negligible sets. The generalized Brenier's theorem \cite[Theorem 10.28]{Villani09}, shows that for any vector $\psi\in \R^N$ the $\mu$-a.e. single valued map $T_\psi: X\to Y$ defined by $T_\psi(x)=y_i$ whenever $x\in \Lag_i(\psi)$, is a minimizer in \eqref{eqn: monge}, from the source measure $\mu$ to the target measure $\nu=\nu_{G(\psi)}$. Clearly $\psi$ and $\psi+r\onevect$ give the same optimal map for any real $r\in\R$. This mapping can be found from the dual Kantorovich problem: in this semi-discrete setting, it is known (see \cite[Chapter 5]{Villani09}) that the minimum value in \eqref{eqn: monge} with $\nu=\nu_\weightvect$ is equal to 	
\begin{align*}
 \max\{-\int_X\phi d\mu-\inner{\psi}{\weightvect}\mid (\phi, \psi)\in L^1(\mu)\times \R^N,\ -\phi(x)-\psi^i\leq c(x, y_i),\ \mu-a.e.\ x\in X\}.
\end{align*}
Then the maximum value is attained by a pair of the form $(\psi^{c^*}, \psi)$ for some $\psi\in \R^N$ and the map $T_\psi$ is the minimizer in \eqref{eqn: monge} between $\mu$ and $\nu_\weightvect$. We will refer to such an $\psi\in \R^N$ and the associated $\psi^{c^*}$ as \emph{an optimal dual vector} and  \emph{an optimal dual potential} for $\nu_\weightvect$.
	
Our first stability result will be stated in terms of the following perturbation in measure:
\begin{defin}
	If $A$, $B\subset \R^n$ are Borel sets, then their \emph{$\mu$-symmetric distance} will be denoted by
	\begin{align}
		\Delta_\mu(A, B):=\mu(A\Delta B)=\mu((A\setminus B)\cup (B\setminus A)).
	\end{align}
\end{defin}
Then our first theorem is:
\begin{thm}\label{thm: symmetric convergence}
	Suppose $c$ satisfies \eqref{Reg} and \eqref{Twist}, and $\mu$ is absolutely continuous. If $\lambda_1$, $\lambda_2\in \weightvectset$ and $\psi_1$, $\psi_2$ are optimal dual vectors for $\nu_{\weightvect_1}$ and $\nu_{\weightvect_2}$ respectively, then
\begin{align}\label{eqn: mu symmetric bound}
 \sum_{i=1}^N \Delta_\mu(\Lag_i(\psi_1), \Lag_i(\psi_2)) \leq 4N\norm{\lambda_1-\lambda_2}_1.
\end{align}
\end{thm}
We point out we make no assumptions on $\mu$ beyond absolute continuity, in particular no geometric assumptions on the support or regularity conditions on the density are made, and the bound is independent of any lower bound on the components of the weight vectors $\weightvect_i$.

Our second stability result on Laguerre cells will be measured in the Hausdorff distance, and will require further conditions on $\mu$ and $c$. On $c$, we need the following condition originally studied by Loeper in \cite{Loeper09}.
\begin{defin}
	We say $c$ satisfies \emph{Loeper's condition} if for each $i\in \{1, \ldots, N\}$ there exists a convex set $X_i\subset \R^n$ and a $C^2$ diffeomorphism $\cExp{i}{\cdot}: X_i\to X$ such that 
	\begin{align*}
		\forall\ t\in\R,\ 1\leq k, i\leq N,\ \{p\in X_i\mid -c(\cExp{i}{p}, y_k)+c(\cExp{i}{p}, y_i)\leq t\}\text{ is convex}.\label{QC}\tag{QC}
	\end{align*}	
	We also say that a set $\tilde{X}\subset X$ is \emph{$c$-convex} with respect to $Y$ if $\invcExp{i}{\tilde{X}}$ is a convex set for every $i\in \{1, \ldots, N\}$. 
\end{defin}
	 \eqref{QC} is a geometric manifestation of the Ma-Trudinger-Wang (MTW) condition which is central to the study of regularity in the Monge-Amp{\`e}re type equation coming from optimal transport. The strong version of the MTW condition was introduced in \cite{MaTrudingerWang05}, and a weak form later in \cite{TrudingerWang09}, both of which deal with higher order regularity for optimal maps in the case of optimal maps between absolutely continuous measures. The results of  \cite{Loeper09} show that if $Y$ is a finite set sampled from from a continuous space, $X$ is $c$-convex with respect to the space $Y$ is sampled from, and $c$ is $C^4$ (along with an analogous convexity condition on the space $Y$ is sampled from), then \eqref{QC} is equivalent to the MTW condition. Additionally, Loeper showed that \eqref{QC} (hence MTW) is necessary for regularity of the optimal transport problem.
\begin{defin}\label{def: universal}
	Suppose $c$ satisfies \eqref{Reg} and \eqref{Twist}, $X$ is a compact set with Lipschitz boundary, $\mu=\rho dx$ for some density $\rho\in C^{0}(X)$, and $\spt\mu\subset X$. Then we will say that a positive, finite constant is \emph{universal} if it has bounds away from zero and infinity depending only on the following quantities: $n$, $\norm{\rho}_{C^0(X)}$, $\mathcal{H}^{n-1}(\partial X)$, $\max_{i\in \{1,\ldots, N\}}\norm{c(\cdot, y_i)}_{C^2(X)}$, and
	\begin{align*}
		\eps_{\mathrm{tw}}&:=\min_{x\in X} \min_{i, j\in \{1, \ldots, N\}, i \neq j}
		\norm{\nabla_xc(x,y_i) - \nabla_xc(x,y_j)},\\ 
		C_\nabla &:=
		\max_{x\in X, i\in \{1, \ldots, N\}} \norm{\nabla_x c(x,y_i)}\\ 
		C_{\exp}
		&:= \max_{i\in \{1, \ldots, N\}} \max\left\{\norm{\exp_i^c}_{C^{0, 1}(\invcExp{i}{X})},
		\norm{(\exp_i^c)^{-1}}_{C^{0, 1}(X)}\right\},\\
		C_{\mathrm{cond}} &:= \max_{i\in \{1,\ldots, N\}} \max_{p \in \invcExp{i}{X}}
		\mathrm{cond}(D \exp_i^c(p)),\\
		C_{\det} &:= \max_{i\in \{1, \ldots, N\}} \norm{\det(D\exp_i^c)}_{C^{0, 1}(\invcExp{i}{X})},
	\end{align*}
	where $\mathrm{cond}$ is the condition number of a linear transformation. These constants are the same as those from \cite[Remark 4.1]{KitagawaMerigotThibert19}.
\end{defin}
\begin{rmk}
If the points $\{y_1,\ldots, y_N\}$ are sampled from some continuous domain $\tilde{Y}$, and $c$ is a cost function on $X\times \tilde{Y}$ satisfying \eqref{Reg}, \eqref{Twist} then all constants in Definition \ref{def: universal}, except $\eps_{\mathrm{tw}}$ are independent of $N$.
\end{rmk}

As for $\mu$, in addition to H\"older regularity of the density, we will require a connectedness assumption on the support.
\begin{defin}\label{def: PW}
	A probability measure $\mu$ on $X$ satisfies a \emph{$(q, 1)$-Poincar\'e-Wirtinger inequality} for some $1\leq q\leq \infty$ if there exists a constant $\Cpw>0$ such that for any $f\in C^1(X)$,
	\begin{align*}
		\norm{f-\int_Xfd\mu}_{L^q(\mu)}\leq \Cpw \norm{\nabla f}_{L^1(\mu)}.
	\end{align*}
For brevity, we will write this as ``$\mu$ satisfies a \emph{$(q, 1)$-PW inequality}''.
\end{defin}
We note that since $X$ has Lipschitz boundary, the class $C^1(X)$ can be unambiguously defined.
\begin{rmk}\label{rmk: PW}
	This condition is used to obtain invertibility of the derivative of the map $G$ in nontrivial directions (see the discussion immediately preceding \cite[Definition 1.3]{KitagawaMerigotThibert19}), and a Poincar{\'e}-Wirtinger inequality can be viewed as a quantitatively strengthened version of connectivity which is sufficient for these purposes. It is classical that if $\rho$ is bounded away from zero on its support, it will satisfy a $(\frac{n}{n-1}, 1)$-PW inequality, and due to scaling $q=\frac{n}{n-1}$ is the largest possible value of $q$ when $\rho$ is continuous.
\end{rmk}
Recall the following definition of Hausdorff distance.
\begin{defin}
 If $x\in \R^n$ and $A\subset \R^n$, we define
\begin{align*}
 d(x, A):=\inf_{y\in A}\norm{x-y}.
\end{align*}
Then for two nonempty sets $A$ and $B\subset \R^n$, the \emph{Hausdorff distance} between $A$ and $B$ is defined by
\begin{align*}
 d_\H(A, B):=\max\(\sup_{x\in A}d(x, B),\ \sup_{x\in B} d(x, A)\).
\end{align*}
\end{defin}
Our second goal is to show stability of the Hausdorff distance between corresponding optimal Laguerre cells, under perturbations of the masses of the target measure. A key ingredient is the following theorem which gives a quantitative Lipschitz bound on the \emph{inverse} of the map $G$, it is here that we critically use the assumption that $q>1$ in the $(q, 1)$-PW inequality for $\mu$.

\begin{thm}\label{thm: quantitative invertibility}
	Suppose that $c$ satisfies \eqref{Reg} and \eqref{Twist}, $X$ has Lipschitz boundary, $\mu=\rho dx$ satisfies a $(q, 1)$-PW inequality with $q>1$, and the map $G$ is differentiable with continuous derivatives. Then for any  $\psi_1$, $\psi_2\in \R^N$ such that $\inner{\psi_1 - \psi_2}{\onevect} = 0$,
	\begin{align*}
	\norm{\psi_1 - \psi_2} \leq \frac{ qN^4C_\nabla\Cpw \norm{ G(\psi_1) - G(\psi_2) }}{4(q-1)\max(\min_i G^i(\psi_1), \min_iG^i(\psi_2))^{1/q} }.
	\end{align*}
\end{thm}
The desired stability result follows as a corollary of this theorem. Specifically, we show nonquantitative stability of the Hausdorff distance of Laguerre cells under a $(1, 1)$-PW inequality on $\mu$, and a local, quantitative estimate of stability under a $(q, 1)$-PW inequality when $q>1$. We carefully note here, for part (2) below it is possible for one of more Laguerre cells for one of either $\psi_1$ or $\psi_2$ to have zero measure, as long as the cells of the other have a strictly positive lower bound.
\begin{cor}\label{cor: hausdorff convergence}
Suppose that $c$ satisfies \eqref{Reg}, \eqref{Twist}, and \eqref{QC}, $X$ is $c$-convex with respect to $Y$, and $\mu=\rho dx$ satisfies a $(q, 1)$-PW inequality with $q\geq 1$.
\begin{enumerate}
\item Suppose $\{\weightvect_k\}_{k=1}^\infty\subset \weightvectset$ converges to some $\weightvect_0\in \weightvectset$ as $k\to 0$, $\psi_k$ and $\psi_0$ are optimal dual vectors for $\nu_{\weightvect_k}$ and $\nu_{\weightvect_0}$ respectively, such that $\inner{\psi_k - \psi_0}{\onevect} = 0$ for all $k$, and $\L(\Lag_i(\psi_0) )>0$ for some $i$. Then
\begin{align*}
		\lim_{k\to 0}d_\H({\Lag_{i}(\psi_k)}, \Lag_i(\psi_0) )=0.
\end{align*}
  \item If $q>1$, there exists a constant $C_1>0$ depending on universal quantities and $q$ with the following property: if $\psi_1$ and $\psi_2$ are optimal dual vectors for the measures $\nu_{\weightvect_1}$ and $\nu_{\weightvect_2}$ respectively, satisfying $\inner{\psi_1 - \psi_2}{\onevect} = 0$, with $\Lag_i(\psi_1)$, $\Lag_i(\psi_2)\neq \emptyset$, and
  \begin{align}\label{eqn: constraint}
C_1 N^{5} \norm{\weightvect_1-\weightvect_2}< \max(\weightvect_1^i, \weightvect_2^i)(\max(\min_{i}\weightvect_1^i, \min_{i}\weightvect_2^i))^{\frac{1}{q}},
\end{align}
then
	\begin{align*}
	d_\H({\Lag_{i}(\psi_1)}, \Lag_{i}(\psi_2) )^n \leq 
	&C_2N^{5}  \norm{\weightvect_1-\weightvect_2}
	\end{align*}
where $C_2>0$ depends on universal constants and the quantities $q$, $\max(\weightvect_1^i, \weightvect_2^i)$, and $\max(\min_{i}\weightvect_1^i, \min_{i}\weightvect_2^i)^{1/q}$.
\end{enumerate}
\end{cor}
\begin{rmk}
	The proof of Corollary \ref{cor: hausdorff convergence} involves a bound on the Lebesgue measure of the symmetric difference of Laguerre cells which could in theory be used to prove the $\mu$-symmetric convergence of the Laguerre cells (as the density of $\mu$ is bounded). However, we opt to present a completely different proof for Theorem \ref{thm: symmetric convergence}, as the method we present here can be applied under less stringent hypotheses. More specifically, in order to exploit the bound on the Lebesgue measure of symmetric difference of cells (Lemma \ref{celldiffbound}) we would require a $(1, 1)$-PW inequality to obtain convergence, and a $(q, 1)$-PW inequality with $q>1$ to obtain a quantitative rate of convergence of the $\mu$-symmetric difference, while our proof of Theorem \ref{thm: symmetric convergence} does not require any kind of PW inequality.
\end{rmk}
\begin{rmk}
 We mention here, there are some practical reasons to consider the stability of Laguerre cells in the Hausdorff distance. The semi-discrete optimal transport problem can be viewed as a model for semi-supervised data clustering: the optimal map assigns to a (continuous) set of data, different clusters with representative data points given by the $y_j$, and the size of each cluster is pre-determined (perhaps empirically, via statistical considerations). The stability in Hausdorff distance then measures the uniform closeness of these clusters with respect to the underlying metric structure, under perturbations of the cluster size.
\end{rmk}

Finally, we can obtain a quantitative estimate of the uniform difference of dual potential functions in terms of the Hausdorff distance of associated Laguerre cells.
\begin{thm}\label{thm: hausdorff and uniform convergence}
 Suppose $c$ satisfies \eqref{Reg}, \eqref{Twist}, and \eqref{QC}, $X$ is bounded and $c$-convex with respect to $Y$. If $\psi_1$, $\psi_2\in \R^N$ are such that $\inner{\psi_1-\psi_2}{\onevect}=0$, there is a universal constant $C>0$ such that 
\begin{align*}
 \norm{\psi_1^{c^*}-\psi_2^{c^*}}_{C^0(X)}\leq \frac{ CN^4\sqrt{\sum_{i=1}^N d_\H(\Lag_i(\psi_1),\Lag_i(\psi_2))^2}}{(\max(\min_i \L(\Lag_i(\psi_1)), \min_i(\L(\Lag_i(\psi_2)))))^{1-\frac{1}{n}}}.
\end{align*}
\end{thm}

\subsection{Outline of the paper}
 In Section \ref{section: mu symmetric convergence}, we use the theory of directed graphs to prove Theorem \ref{thm: symmetric convergence} on the $\mu$-symmetric convergence of Laguerre cells. In Section \ref{section: inj of G}, we establish some preliminary invertibility properties of the mapping $G$ under our setting. In Section \ref{section: quantitative invertibility} we prove the quantitative invertibility result Theorem \ref{thm: quantitative invertibility}, this is done via some alternative spectral estimates of the transformation $DG$ which are of independent interest. In Section \ref{section: hausdorff estimates} we gather some estimates on the Hausdorff measure of differences of Laguerre cells, mostly using convex geometry, and then prove Corollary \ref{cor: hausdorff convergence}. Finally, Section \ref{sec: uniform convergence} gathers the results needed to prove the estimate Theorem \ref{thm: hausdorff and uniform convergence}. In each section, we progressively add more conditions on $c$, $\mu$, and $X$, which are detailed there.

\subsection{Literature analysis}
One can use \cite[Corollary 5.23]{Villani09} to see if $\mu_k$ and $\nu_k$ weakly converge to some probability measures, $c$ satisfies \eqref{Reg} and \eqref{Twist}, and the limit of the sequence $\{\mu_k\}$ is absolutely continuous, then the sequence of optimal transport maps minimizing \eqref{eqn: monge} converge in measure to the optimal transport map of the limiting problem, however there is no explicit rate of convergence. Currently there are few results with quantitative rates:  quantitative $L^2$ stability of the transport maps (equivalent to $H^1$ convergence of dual potentials) is shown under discretization of the target measure in \cite{Berman18} and for general perturbations in the 2-Wasserstein metric of the target measure in \cite{MerigotDelalandeChazal19}. These results do give our convergence result in $\mu$-symmetric measure, however the discussion in \cite{Berman18} and \cite{MerigotDelalandeChazal19} are restricted to quadratic distance squared cost, and have more stringent conditions on the source measure $\mu$ than our result. Under conditions which yield regular optimal transport maps, \cite{Gigli11} shows if $\mu_t$ is an absolutely continuous curve of probability measures with respect to the $\mathcal{W}_p$ optimal transport metric, then the curve of optimal transport maps is H\"older continuous, measured in $L^2(\mu)$. The result in the case $p=2$ is originally due to Ambrosio (also reported in \cite{Gigli11}). Finally, \cite[Theorem 3.1]{AmbrosioGlaudoTrevisan19} is a quantitative result for optimal transport with geodesic distance squared cost on compact manifolds (again, in $L^2$ difference of transport maps). There seem to be no results with rates for uniform convergence.

\subsection*{Acknowledgements} The authors would like to thank Filippo Santambrogio for pointing out the relationship between the convergence in Hausdorff distance of Laguerre cells, and the uniform convergence of the dual potential functions.

The authors would also like to thank the anonymous referee for a thorough and careful reading of the paper which has lead to great improvements in the presentation. In particular, they would like to thank the referee for pointing out the bound in Theorem \ref{thm: HausPsiBound} can be significantly improved for the case of inner product cost, which is now added in Remark \ref{rmk: optimality}.

\section{$\mu$-symmetric convergence of Laguerre cells}\label{section: mu symmetric convergence}
For the remainder of the paper, we assume that $c$ satisfies \eqref{Reg}, \eqref{Twist}, and $\mu$ is absolutely continuous. In this section, we do not assume \eqref{QC} or any  regularity on the density of $\mu$.

We will actually prove our first stability result Theorem \ref{thm: symmetric convergence} for a variant of the optimal transport problem first dealt by the authors in \cite{BansilKitagawa19a}\footnote{The resulting proof is only slightly more involved than in the classical case, we have opted to prove our results in this setting for use in a forthcoming work on numerics.} (a specific case of the problem also appears in \cite{CrippaJimenezPratelli09} in the context of queue penalization). In addition to the setting of the semi-discrete optimal transport problem \eqref{eqn: monge}, we assume there is a \emph{storage fee function} $F: \R^N\to\R$. Then the \emph{semi-discrete optimal transport with storage fees} is to find a pair $(T, \weightvect)$ with $\weightvect=(\weightvect^1,\ldots, \weightvect^N)\in \R^N$ and $T: X\to Y$ measureable satisfying
\begin{align*}
	T_\#\mu = \sum_{i=1}^N \weightvect^i \delta_{y_i}
\end{align*}
such that
\begin{align}\label{eqn: monge ver}
	\int_X c(x, T(x)) d\mu + F(\weightvect) = \min_{\tilde \weightvect\in \R^N,\ \tilde{T}_\#\mu = \sum_{i=1}^N \tilde\weightvect^i \delta_{y_i}} \int_X c(x, \tilde{T}(x)) d\mu + F(\tilde\weightvect).
\end{align}
For this section, we will suppose $F_1$, $F_2: \R^N\to \R\cup\{+\infty\}$ are two proper convex functions equal to $+\infty$ outside of $\weightvectset$. Under our assumptions on $\mu$ and $c$, by \cite[Theorem 2.3 and Proposition 3.5]{BansilKitagawa19a} there exist pairs $(T_1, \weightvect_1)$ and $(T_2, \weightvect_2)$ minimizing \eqref{eqn: monge ver} with storage fee functions equal to $F_1$ and $F_2$ respectively, along with (see \cite[Theorem 4.7]{BansilKitagawa19a}) vectors $\psi_1$, $\psi_2\in \R^N$ such that $G(\psi_1)=\weightvect_1$, $G(\psi_2)=\weightvect_2$.

Also given any set $A$, we write $\ind(x\mid A):=
\begin{cases}
0,&x\in A,\\
+\infty,&x\not\in A,
\end{cases}$ for the \emph{indicator function} of the set $A$, and for any vector $w\in\R^N$ with nonnegative entries, we denote $F_w:=\sum_{i=1}^N\ind(\cdot\mid [0, w^i])=\ind(\cdot\mid \prod_{i=1}^N[0, w^i])$.
\subsection{The Exchange Digraph}
We now define a weighted directed graph (digraph), $D$, as follows. The vertex set is $y_1, \dots, y_N$. When $i \neq j$, there is a directed edge from $y_i$ to $y_j$ if $\mu({\Lag_i(\psi_1)}\cap {\Lag_j(\psi_2)}) > 0$, and in this case that edge is assigned weight $\mu({\Lag_i(\psi_1)}\cap {\Lag_j(\psi_2)})$. We denote the weight of an edge $e$ by $w(e)$. 

Essentially this digraph keeps track of how much mass is shifted from one Laguerre cell to a different one under a change of the storage fee function. Indeed note that $\weightvect_2^i = \weightvect_1^i - \deg^+(y_i) + \deg^-(y_i)$ where 
\begin{align*}
 \deg^+(y_i):&=\sum_{\{e\mid e\text{ is directed out from }y_i\}}w(e),\\
 \deg^-(y_i):&=\sum_{\{e\mid e\text{ is directed into }y_i\}}w(e),
\end{align*}
denote outdegree and indegree respectively. 

First we use an argument reminiscent of the $c$-cyclical monotonicity of optimal transport plans to prove the following lemma. We comment that the following lemma does not involve the storage fees $F_1$ and $F_2$, and can be proved entirely in the context of classical semi-discrete optimal transport theory.
\begin{lem}\label{acyclic}
	$D$ is acyclic
\end{lem}

\begin{proof}
	Suppose for sake of contradiction there exists a cycle $y_{i_1}$, $e_1$, $y_{i_2}, \ldots, y_{i_l}$, $e_l$, $y_{i_{l+1}}$ where $i_{l+1} = i_1$ and  $e_j$ is a directed edge from $y_{i_{j}}$ to $y_{i_{j+1}}$. Let $m_0:=\min_{1\leq j\leq l}w(e_j)>0$, then for each $1\leq j\leq l$ there exists a measurable set $A_j\subset \Lag_{i_j}(\psi_1)\cap \Lag_{i_{j+1}}(\psi_2)$ with $\mu(A_j)=m_0$, and we define $A_{l+1}=A_1$.

	Now define the sets $\{\ti{C}_k\}_{k=1}^N$ by 
	\begin{align}
		\ti{C}_k=
		\begin{cases}
			(\Lag_{i_{j+1}}(\psi_2)\cup A_{j+1}) \setminus A_{j},& k=i_{j+1},\ 1\leq j\leq l,\\
			\Lag_k(\psi_2),&k \not\in \{ i_1, \dots, i_l \},
		\end{cases}
	\end{align}
	and the map $\tilde{T}: X\to Y$ defined by $\tilde{T}(x)=\sum_{k=1}^N y_k\mathds{1}_{\ti{C}_k}(x)$. Since $\Lag_i(\psi_1)$ and $\Lag_j(\psi_1)$ are disjoint up to sets of $\mu$ measure zero for $i\neq j$, we must have that the sets $A_j$ are mutually disjoint up to $\mu$ measure zero sets, thus $\ti T_\#\mu=\sum_{k=1}^N \mu(\ti{C}_k)\delta_{y_k}=\sum_{k=1}^N \weightvect_2^k\delta_{y_k}$ but $\ti T\neq T_2$ on a set of positive $\mu$ measure. It is clear that $T_2$ is an optimal map minimizing the classical optimal transport problem \ref{eqn: monge} with target measure $\nu_{\weightvect_2}$, which is uniquely determined $\mu$-a.e. by \cite[Theorem 10.28]{Villani09}. Thus we have
	\begin{align*}
	\sum_{k=1}^N \int_{ \ti{C}_k  } c(x, y_{ k } ) d\mu(x)> \sum_{k=1}^N\int_{  {\Lag_{k}(\psi_2)}  } c(x, y_{k} ) d\mu(x).
	\end{align*}
	Hence,
	\begin{align}
		0&<\sum_{k=1}^N \int_{ \ti{C}_k  } c(x, y_{ k } ) d\mu(x) - \sum_{k=1}^N\int_{  {\Lag_{k}(\psi_2)}  } c(x, y_{k} ) d\mu(x)\notag\\
		&=\sum_{k=1}^N \int_{ \Lag_{k}(\psi_2)  } c(x, y_{ k } ) d\mu(x) - \sum_{k=1}^N\int_{  \Lag_{k}(\psi_2)  } c(x, y_{k} ) d\mu(x)\notag\\
		&+\sum_{j=1}^{l-1} \(\int_{ A_{j+1}  } c(x, y_{  i_{j+1} } ) d\mu(x) - \int_{  A_{j} } c(x, y_{  i_{j+1} } ) d\mu(x)\)\notag\\
		&=\sum_{j=1}^{l-1} \(\int_{ A_{j+1} } c(x, y_{  i_{j+1} } ) d\mu(x) - \int_{  A_{j} } c(x, y_{  i_{j+1} } ) d\mu(x)\)\label{eqn: mono contradiction}.
	\end{align}
	On the other hand, defining the sets $\{\ti{D}_k\}_{k=1}^N$ by 
	\begin{align}
		\ti{D}_k=
		\begin{cases}
			(\Lag_{i_{j+1}}(\psi_1)\cup A_{j})\setminus A_{j+1} ,& k=i_{j+1},\ 1\leq j\leq l,\\
			\Lag_k(\psi_1),&k \not\in \{ i_1, \dots, i_l \},
		\end{cases}
	\end{align}
	and taking the map $\widehat{T}(x)=\sum_{k=1}^N y_k\mathds{1}_{\ti D_k}(x)$, we can make an analogous calculation which yields the opposite inequality as \eqref{eqn: mono contradiction}, giving a contradiction.	
\end{proof}

For the next three Lemmas \ref{L2}, \ref{L3}, and \ref{L4}, we shall be concerned about the case where
\begin{align}
	F_1(\weightvect) &= \sum_{i=1}^N \ind(\weightvect^i\mid [a^i, b^i]), \notag\\
	F_2(\weightvect) &= \ind(\weightvect^1\mid [a^1, b^1+\pert]) + \sum_{i=2}^N \ind(\weightvect^i\mid [a^i, b^i])\label{eqn: choices of F}
\end{align}
where $a^i \leq b^i$ and $\sum a^i \leq 1 \leq \sum b^i$. Recall that $(T_1, \weightvect_1)$, $(T_2, \weightvect_2)$ are the minimizers in \eqref{eqn: monge ver} associated with $F_1$, $F_2$ respectively; in particular we must have $a^i\leq \weightvect_1^i\leq b^i$ for all $i\in \{1, \ldots, N\}$, $a^1\leq \weightvect_2^1\leq b^1+\pert$, and $a^i\leq \weightvect_2^i\leq b^i$ for all $2\leq i\leq N$.

\begin{lem} \label{L2}
	
	Suppose we take $F_1$ and $F_2$ as in \eqref{eqn: choices of F} and there exists some vertex $y_m$ of $D$ with an incoming edge. Then $\weightvect^m_1 = b^m$. 
	
\end{lem}

\begin{proof}
	
	Let $i_1 = m$. Suppose the incoming edge, which we denote $e_1$, goes from $y_{i_2}$ to $y_{i_1}$. We claim that there is a path $P = (y_{i_1}, e_1, y_{i_2}, \ldots, y_{l-1}, e_{l-1}, y_{i_l})$, where $e_j$ is an edge from $y_{i_{j+1}}$ to $y_{i_{j}}$, such that the last vertex $y_{i_l}$ has no incoming edges. 
	
	We construct such a path recursively. Let $P_1 = (y_{i_1}, e_1, y_{i_2})$ and suppose that 
	
	$P_r = (y_{i_1}, e_1, y_{i_2}, \ldots, y_{i_r}, e_{r}, y_{i_{r+1}})$ has been constructed. If $y_{i_{r+1}}$ has no incoming edges then $P_r$ is the desired path and we are done. If not $y_{i_{r+1}}$ has an incoming edge which we denote $e_{r+1}$. Let $y_{i_{r+2}}$ be the originating vertex of $e_{r+1}$ and let $P_{r+1} = (y_{i_1}, e_1, y_{i_2},  \ldots, y_{i_{r+1}}, e_{r+1}, y_{i_{r+2}})$.

	If the above process does not terminate then since we only have finitely many vertices we must eventually repeat a vertex, i.e. there is $r > j$ so that ${i_j} = i_r$. However this means that $P_{r}$ contains a cycle which contradicts Lemma \ref{acyclic} above. 
	
	Now let $m_0 = \min(b^m- \weightvect^m_1, w(e_1), \dots, w(e_{l-1}))$. Suppose for sake of contradiction that $\weightvect^m_1 < b^m$, then $m_0 > 0$. Note that	
	\begin{align}
		\weightvect_2^{i_l}&=\weightvect_1^{i_l}-\deg^+(y_{i_l})+\deg^-(y_{i_l})\leq b^{i_l}-w(e_{l-1})+0\leq b^{i_l}-m_0.\label{eqn: k weight bound}
	\end{align}

	Now just as in the proof of Lemma \ref{acyclic} for $j \in \{2, \dots, l \}$ there exist sets $A_j$ so that $A_j \subset \Lag_{i_{j}}(\psi_1)\cap \Lag_{i_{j-1}}(\psi_2)$, and $\mu(A_j) = m_0$. We define $A_1 = A_{l+1} = \emptyset$. 
	Now define the sets $\{\ti C_k\}_{k=1}^N$ by 
	\begin{align}
		\ti C_k=
		\begin{cases}
			(\Lag_{i_j}(\psi_2)\cup A_{j}) \setminus A_{j+1},& k=i_j, j \in \{1, \dots, l \},\\
			\Lag_k(\psi_2),&k \not\in \{ i_1, \dots, i_{l} \}.
		\end{cases}
	\end{align}
	and the map $\tilde{T}: X\to Y$ defined by $\tilde{T}(x)=\sum_{k=1}^N y_k\mathds{1}_{\ti C_k}(x)$. %
	Just as in the proof of Lemma \ref{L2} above, we have $\ti T_\#\mu=\sum_{k=1}^N \mu(\ti {C}_k)\delta_{y_k}$ and $\ti T\neq T_2$ on a set of positive $\mu$ measure (however, note that we do not have $\mu(\ti{C}_k)=\weightvect_2^k$ for $k=i_1$, $i_l$). Since $(T_2, \weightvect_2)$ is the unique minimizer of \eqref{eqn: monge ver} with storage fee function $F_2$ by \cite[Corollary 4.5]{BansilKitagawa19a}, we must have
	\begin{align*}
	\sum_{k=1}^N \int_{ \ti C_k  } c(x, y_{ k } ) d\mu(x) + F_2((\mu(\ti {C}_1), \ldots, \mu(\ti {C}_N))) > \sum_{k=1}^N\int_{  {\Lag_{k}(\psi_2)}  } c(x, y_{k} ) d\mu(x) + F_2(\weightvect_2).
	\end{align*}
	However now note that
	\begin{align*}
	\mu(\ti {C}_k)
	= \begin{cases}
	\lambda_2^{i_1} - m_0, & k = i_1, \\
	\lambda_2^{i_l} + m_0, & k = i_l, \\
	\lambda_2^k, & \text{else}.
	\end{cases}
	\end{align*}
	By \eqref{eqn: k weight bound}, we have that $\mu(\ti {C}_{i_l}) =\lambda_2^{i_l} + m_0\leq b^{i_l}$. Also for $k \neq i_l$ we have $\mu(\ti {C}_{k}) \leq \lambda_2^k \leq b^{k}$, hence $F_2((\mu(\ti {C}_1), \ldots, \mu(\ti {C}_N))) = 0$. Thus the above becomes
	\begin{align}
	\sum_{k=1}^N \int_{ \ti C_k  } c(x, y_{ k } ) d\mu(x) > \sum_{k=1}^N\int_{  {\Lag_{k}(\psi_2)}  } c(x, y_{k} ) d\mu(x),\label{eqn: cyclical inequality}
	\end{align} 
	and by a calculation identical to the one leading to \eqref{eqn: mono contradiction}, we have
	\begin{align*}
	0%
	&<\sum_{j=1}^{l} \(\int_{  A_{j}  } c(x, y_{  i_{j} } ) d\mu(x) - \int_{  A_{j+1}  } c(x, y_{  i_{j} } ) d\mu(x)\).
	\end{align*}
	On the other hand, define the sets $\{\ti D_k\}_{k=1}^N$ by 
	\begin{align}
		\ti D_k=
		\begin{cases}
			(\Lag_{i_j}(\psi_1)\cup A_{j+1}) \setminus A_{j},& k=i_j,\ j \in \{1, \dots, l \},\\
			\Lag_k(\psi_1),&k \not\in \{ i_1, \dots, i_{l} \}.
		\end{cases}
	\end{align}
	Note that 
 	\begin{align*}
	\mu(\ti {D}_k)
	= \begin{cases}
	\lambda_1^{i_1} + m_0, & k = i_1 \\
	\lambda_1^{i_l} - m_0, & k = i_l \\
	\lambda_1^k, & \text{else}.
	\end{cases}
	\end{align*}
	By definition of $m_0$ we have $m_0 \leq b^{m} - \lambda_1^{m}=b^{i_1} - \lambda_1^{i_1}$, hence we have $\mu(\ti{D}_{i_1}) \leq b^{i_1}$. Thus as above, $F_2((\mu(\ti {D}_1), \ldots, \mu(\ti {D}_N))) = 0$ and a similar argument yields the opposite inequality of \eqref{eqn: cyclical inequality} to obtain a contradiction. 
\end{proof}

\begin{lem} \label{L3}
	
	Suppose we take $F_1$ and $F_2$ as in \eqref{eqn: choices of F}. Then for $i \neq 1$, $\weightvect_2^i \leq \weightvect_1^i$. Furthermore, if $y_i$ has an incoming edge it must have an outgoing edge. Finally, $y_1$ has no outgoing edges. 
	
\end{lem}

\begin{proof}
	
	Recall that $\weightvect_2^i = \weightvect_1^i - \deg^+(y_i) + \deg^-(y_i)$. 
	
	Suppose $i\neq 1$. If $y_i$ has no incoming edges then $\deg^-(y_i) = 0$ so $\weightvect_2^i = \weightvect_1^i - \deg^+(y_i) \leq \weightvect_1^i$. If $y_i$ has at least one incoming edge then $\weightvect_1^i = b^i$ by Lemma \ref{L2} above. Since $i \neq 1$ and $F_2(\weightvect_2) < +\infty$, we must have $\weightvect_2^i \leq b^i$. In either case $\weightvect_2^i \leq \weightvect_1^i$. 
	
	Now if $y_i$ has an incoming edge then 
	\begin{align*}
		\deg^+(y_i) = \weightvect_1^i - \weightvect_2^i + \deg^-(y_i) \geq \deg^-(y_i) > 0,
	\end{align*}
	so there must be an outgoing edge.
	
	Finally suppose for sake of contradiction that $y_1$ has an outgoing edge. We recursively construct a path similar to that in the proof of Lemma \ref{L2}. 	
	Set $i_1 = 1$, $P_1 = (y_{i_1}, e_1, y_{i_2})$ and suppose that 
	$P_l = (y_{i_1}, e_1, y_{i_2}, \ldots, y_{i_{l}}, e_{l}, y_{i_{l+1}})$ has been constructed where $e_j$ is an edge directed from $y_{i_{j}}$ to $y_{i_{j+1}}$. If $y_{i_{l+1}} = y_{i_1}$ then we have constructed a cycle which contradicts Lemma \ref{acyclic}. If $y_{i_{l+1}} \neq y_{i_1} = y_1$, then $y_{i_{l+1}}$ has an outgoing edge which we denote $e_{l+1}$. Set $y_{i_{l+2}}$ to be the tail of $e_{l+1}$ and let $P_{l+1} = (y_{i_1}, e_1, y_{i_2}, \ldots, y_{l}, e_{l+1}, y_{i_{l+2}})$. 
	Since we only have finitely many vertices the above process must repeat a vertex which will produce a cycle. This contradicts Lemma \ref{acyclic} hence $y_1$ cannot have any outgoing edges. 
\end{proof}

\begin{rmk}
	
	Recall that in an directed acyclic graph the vertices can be given an ordering, called a topological ordering, so that every edge goes from a vertex with smaller index to a vertex with larger index. See \cite[Proposition 2.1.3]{BangJensenGutin09} and the associated footnote for more details.  
	
\end{rmk}

\begin{lem} \label{L4}
	Suppose again we take $F_1$ and $F_2$ as in \eqref{eqn: choices of F}. Then every edge has outdegree at most $\pert$, in particular every vertex has weight at most $\pert$. In this case we have $\norm{\weightvect_1 - \weightvect_2}_1 \leq 2\pert$ and $\sum_{i=1}^N \Delta_\mu({\Lag_{i}(\psi_1)}, {\Lag_{i}(\psi_2)}) \leq 2N\pert$. 
\end{lem}

\begin{proof}
	
	Let $y_{i_1}, \dots, y_{i_N}$ be a topological ordering. By Lemma \ref{L3} we may assume $i_N = 1$. Consider the function
	\begin{align*}
	f(k) 
	= \sum_{j=1}^k \deg^+(y_{i_j}) - \deg^-(y_{i_j}) 
	= \sum_{j=1}^k \weightvect_1^{i_j} - \weightvect_2^{i_j}
	\end{align*}
	for $k \leq N-1$. 
	
	By Lemma \ref{L3} $f$ is increasing. Let $E_k$ be the collection of edges directed from one of the vertices $y_{i_1}, \dots, y_{i_k}$ and into one of the vertices $y_{i_{k+1}}, \dots, y_{i_N}$. Then we have
	\begin{align*}
	f(k) = \sum_{e \in E_k} w(e);
	\end{align*}
as we have imposed a topological ordering, there is no edge directed from one of the vertices $y_{i_{k+1}}, \dots, y_{i_N}$ to one of the vertices $y_{i_1}, \dots, y_{i_k}$. In particular $f(k) \geq \deg^+(y_{i_k})$, thus $f(N-1) \geq \deg^+(y_{i_k})$ for all $k\leq N-1$. Note that $E_{N-1}$ is the collection of all edges directed to $y_{i_N}=y_{1}$. Hence
	\begin{align*}
	\deg^+(y_{i_k}) \leq f(N-1) = \sum_{e \in E_{N-1}} w(e) = \deg^-(y_{1}).
	\end{align*}
	If $y_1$ has no incoming edges then this gives us $\deg^+(y_{i_k}) = 0$. Otherwise by Lemma \ref{L2}
	\begin{align*}
	\deg^-(y_{1}) = \weightvect_2^1 - \weightvect_1^1 + \deg^+(y_1) = \weightvect_2^1 - b^1
	\end{align*}
	where $\deg^+(y_1) = 0$ by Lemma \ref{L3}. Since $F_2(\weightvect_2)<+\infty$, we must have $\weightvect_2^1 \leq b^1 + \pert$ hence each vertex has outdegree at most $\pert$. 
	
	Next by Lemma \ref{L3}, $\weightvect^i_2 \leq \weightvect^i_1$ for $i \neq 1$, since $\weightvect_1$, $\weightvect_2\in \weightvectset$ this implies $\weightvect^1_2 \geq \weightvect^1_1$. Hence
	\begin{align*}
	\norm{\weightvect_1 - \weightvect_2}_1 
	&= \sum_{i=1}^N \abs{\weightvect_2^i - \weightvect_1^i} \\
	&= \weightvect^1_2 - \weightvect^1_1 + \sum_{i=2}^N (\weightvect_1^i - \weightvect_2^i) \\
	&= \weightvect^1_2 - \weightvect^1_1 + (1- \weightvect_1^1) - (1- \weightvect_2^1) \\
	&= 2(\weightvect^1_2 - \weightvect^1_1) =2 (\deg^-(y_{1})-\deg^+(y_{1})) \\
	&\leq 2\pert
	\end{align*}
	where we have used $\sum_{i=1}^N \weightvect_1^i = \sum_{i=1}^N \weightvect_2^i = 1$. 
	
	Next we have
	\begin{align*}
	&\mu({\Lag_{i}(\psi_1)} \setminus {\Lag_{i}(\psi_2)} )
	= \mu({\Lag_{i}(\psi_1)} \cap ({\Lag_{i}(\psi_2)})^c) \\
	&= \mu({\Lag_{i}(\psi_1)} \cap \bigcup_{j\neq i}^N {\Lag_{j}(\psi_2)}) 
	= \sum_{j\neq i}^N \mu({\Lag_{i}(\psi_1)} \cap {\Lag_{j}(\psi_2)})
	= \deg^+(y_i) 
	\leq \pert
	\end{align*}
	and so $\sum_{i=1}^N \mu({\Lag_{i}(\psi_1)} \setminus {\Lag_{i}(\psi_2)} ) \leq N\pert$. A similar argument gives
	\begin{align*}
	\sum_{i=1}^N \mu({\Lag_{i}(\psi_2)} \setminus {\Lag_{i}(\psi_1)} ) = \sum_{i=1}^N \deg^-(y_i) = \sum_{i=1}^N \deg^+(y_i) \leq N\pert
	\end{align*}
where the final equality comes from
\begin{align*}
\sum_{i=1}^N \deg^-(y_i) =\sum_{i=1}^N (\deg^+(y_i)+\lambda_2^i-\lambda_1^i)=\sum_{i=1}^N \deg^+(y_i),
\end{align*}
finishing the proof.	
\end{proof}
By repeated applications of the Lemma above, we can analyze the digraph $D$ when $F_1$ and $F_2$ are characteristic functions of two different hyperrectangles.
\begin{thm} \label{T5}
	
	Suppose we have
	\begin{align*}
	F_1(\weightvect) &= \sum_{i=1}^N \ind(\weightvect^i\mid [a_1^i, b_1^i]), \\
	F_2(\weightvect) &= \sum_{i=1}^N \ind(\weightvect^i\mid [a_2^i, b_2^i]).
	\end{align*}
Then $\norm{\weightvect_1 - \weightvect_2}_1 \leq 2 (\norm{a_1 - a_2}_1 + \norm{b_1 - b_2}_1)$ and $\sum_{i=1}^N \Delta_\mu({\Lag_{i}(\psi_1)}, {\Lag_{i}(\psi_2)}) \leq 2 N (\norm{a_1 - a_2}_1 + \norm{b_1 - b_2}_1)$.
	
\end{thm}

\begin{proof}
	The estimate can be seen by applying Lemma \ref{L4} and perturbing the initial rectangle defined by $a_1$ and $b_2$, one coordinate at a time. If $a_1 = a_2$ then this follows from induction on the number of equal terms in $b_1$, $b_2$, repeatedly applying Lemma \ref{L4}, and the triangle inequality. The case $a_1 \neq a_2$ is handled with a symmetric argument and the triangle inequality. 
\end{proof}
\begin{rmk}\label{rmk: symm diff optimal}
 The first estimate from Theorem \ref{T5} above is sharp, and the second is almost sharp (up to replacing the constant $2N$ by $2N-2$). Let $X=[0, N]$ and $\mu$ be Lebesgue measure restricted to $X$ and normalized to unit mass, and take the cost function $c(x, y)=-xy$. Fix any $N>1$ and let $y_i=i-\frac{1}{2}\in \R$ for $i\in \{1, \ldots, N\}$. We take the functions
\begin{align*}
 F_1(\lambda):&= \sum_{i=1}^{N} \ind(\weightvect^i\mid [0, \frac{1}{N}])=\ind(\weightvect \mid [0, \frac{1}{N}]^N),\\
 F_2(\lambda):&= \sum_{i=1}^{N-1} \ind(\weightvect^i\mid [0, \frac{1}{N}])+\ind(\weightvect^N \mid [0, \frac{2}{N}])=\ind(\weightvect \mid [0, \frac{1}{N}]^{N-1}\times [0, \frac{2}{N}]),
\end{align*}
then note that in the notation of Theorem \ref{T5},
\begin{align*}
 \norm{a_1 - a_2}_1 + \norm{b_1 - b_2}_1=\frac{1}{N}.
\end{align*}
The optimal transport problem with storage fee $F_1$ is actually a classical optimal transport problem, and it is not difficult to see that if the pair $(\psi_1, \lambda_1)$ yields a minimizer, then the associated Laguerre cells are given by $\Lag_i(\psi_1)=[i-1,i]$ for $i\in \{1, \ldots, N\}$ with $\lambda_1=(\frac{1}{N}, \ldots, \frac{1}{N})$. On the other hand, we claim that for the problem with storage fee $F_2$, an optimal pair $(\psi_2, \lambda_2)$ is given by
 \begin{align*}
 \lambda_2:&=(0, \frac{1}{N}, \ldots, \frac{1}{N}, \frac{2}{N}),\\
 \psi_2^i:&=
\begin{cases}
 0,&i=1,\\
 \frac{(i-1)(i-2)}{2},&i\in \{2, \ldots, N\},
\end{cases}\\
\Lag_{i}(\psi_2) :&=
\begin{cases} 
\{0\},  &i = 1, \\
[i-2, i - 1 ], &i \in {2, \dots, N-1} \\
[N-2, N ], &i =N.
\end{cases}
\end{align*}
Indeed, note we may replace $F_2$ by $F_2+\ind(\cdot\mid \Lambda)$ without changing the optimizer. Then for any $\lambda\in ([0, \frac{1}{N}]^{N-1}\times [0, \frac{2}{N}])\cap \Lambda$, we find
\begin{align*}
 F_2(\lambda_2)+\inner{\lambda-\lambda_2}{\psi_2}&=\sum_{i=2}^{N-1} (\lambda^i-\frac{1}{N})\psi_2^i+(\lambda^N-\frac{2}{N})\psi_2^N\leq 0= F_2(\lambda)
\end{align*}
while if $\lambda \in \R^N\setminus (([0, \frac{1}{N}]^{N-1}\times [0, \frac{2}{N}])\cap \Lambda)$ we have $F_2(\lambda)=+\infty$, hence $\psi_2\in \subdiff{F_2}{\lambda_2}$. By \cite[Theorem 4.7]{BansilKitagawa19a}, this shows the optimality of $(\psi_2, \lambda_2)$. Now we can calculate,
\begin{align*}
 \norm{\lambda_1-\lambda_2}_1=\frac{2}{N}=2 (\norm{a_1 - a_2}_1 + \norm{b_1 - b_2}_1)
\end{align*}
while
\begin{align*}
 \sum_{i=1}^N \Delta_\mu({\Lag_{i}(\psi_1)}, {\Lag_{i}(\psi_2)})&=\frac{1}{N}+\frac{2}{N}(N-2)+\frac{1}{N} = (2N-2) (\norm{a_1 - a_2}_1 + \norm{b_1 - b_2}_1).
\end{align*}
A simple product construction can be used to easily adapt this example to $\R^n$ for $n>1$.
\end{rmk}
We now show a version of Theorem \ref{thm: symmetric convergence} which applies to the more general setting of optimal transport with storage fees, and our main theorem will follow immediately.
\begin{cor} \label{C6}
	
	Suppose that $F_1$, $F_2: \R^N\to \R\cup\{+\infty\}$ are two proper convex functions equal to $+\infty$ outside of $\weightvectset$. Then
	\begin{align*}
	\sum_{i=1}^N \Delta_\mu({\Lag_{i}(\psi_1)}, {\Lag_{i}(\psi_2)}) \leq 4 N \norm{\weightvect_1 - \weightvect_2}_1.
	\end{align*} 
\end{cor}
\begin{proof}
	Define
	\begin{align*}
	\ti F_1(\weightvect) &= \sum_{i=1}^N \ind(\weightvect^i\mid [a_1^i, b_1^i]), \\
	\ti F_2(\weightvect) &= \sum_{i=1}^N \ind(\weightvect^i\mid [a_2^i, b_2^i]),
	\end{align*}
	where $a_1^i = b_1^i = \weightvect_1^i$ and $a_2^i = b_2^i = \weightvect_2^i$. We see that if $(\ti{T}_1, \ti{\weightvect}_1)$, $(\ti{T}_2, \ti{\weightvect}_2)$ are minimizers for \eqref{eqn: monge ver} with storage fee functions $\ti F_1$ and $\ti F_2$, then up to sets of $\mu$ measure zero $\ti{T}_1^{-1}(\{y_i\}) = {\Lag_i(\psi_1)}$ and $\ti{T}_2^{-1}(\{y_i\}) = {\Lag_i(\psi_2)}$ for each $i\in \{1, \ldots, N\}$. Hence the result follows from applying Theorem \ref{T5} to $\ti F_1, \ti F_2$. 
\end{proof}

\begin{proof}[Proof of Theorem \ref{thm: symmetric convergence}]
	
	By taking $F_1, F_2$ to be the indicator functions for two points in $\weightvectset$, the above corollary immediately yields the theorem.

\end{proof}

\section{Injectivity of G}\label{section: inj of G}
In this section we prove non-quantitative invertibility of $G$ as preparation for the quantitative invertibility result Theorem \ref{thm: quantitative invertibility}. Starting in this section, in addition to all previous assumptions, we add that $\mu=\rho dx$ where $\rho\in C^0(X)$ and $\mu$ satisfies a $(1, 1)$-PW inequality, and also assume $X$ is a compact set with Lipschitz boundary, such that $\spt\mu\subset X$. We mention that the results of this section do not require the assumption \eqref{QC}.

\begin{rmk}
We remark here that the ultimate goal, Proposition \ref{homeomorphism} follows if the set $\{\rho>0\}$ is connected, by \cite[Proposition 4.1]{Loeper09}. On the other hand, if $\mu$ satisfies a $(1, 1)$-PW inequality, it is easy to see that $\spt\rho$ is connected (if $\spt\rho$ is disconnected, taking $f$ in Definition \ref{def: PW} equal to two different constants on two connected components will cause the inequality to fail). However, we were unable to prove the desired injectivity under this weaker connectedness, while it is also not clear if a $(1, 1)$-PW inequality implies the stronger connectedness of $\{\rho>0\}$. Thus we have opted to prove the claims in this section under the assumption of a $(1, 1)$-PW inequality.
\end{rmk}
\begin{defin}\label{def: pseudo c-transforms}
	If $\varphi: X\to \R\cup \{+\infty\}$  (not identically $+\infty$), its pseudo $c$-transform is a vector $\varphi^{c^\dagger}\in \R^N$, defined by
	\begin{align*}
		(\varphi^{c^\dagger})^i:=\sup_{x\in \spt \mu}(-c(x, y_i)-\varphi(x)).
	\end{align*}
Also let $\Psi_c = \{\psi \in \R^n: \psi = \psi^{c^*c^\dagger} \}$. 
\end{defin}

\begin{lem}\label{lem: lag cell lemma}
Suppose $\psi_1$, $\psi_2\in \R^N$ are such that $\weightvect:=G(\psi_1)=G(\psi_2)$, and suppose that $\lambda^i > 0$ for some index $i$. If $x \in \interior (X)\cap \Lag_i(\psi_1)$ and $\rho(x) > 0$ then $x \in \Lag_i(\psi_2)$. 
\end{lem}

\begin{proof}
Suppose by contradiction, for such an $x$ we have $x \not\in \Lag_i(\psi_2)$. As the zero set of a continuous function $\Lag_i(\psi_2)$ is closed, hence there is a neighborhood of $x$ in $X$, say $U$, so that $U \cap \Lag_i(\psi_2) = \emptyset$. Next since $\rho(x) > 0$, by continuity of $\rho$ there is an open neighborhood of $x$, say $V\subset U$ so that $\rho > 0$ on $V$.

Note that $\Lag_i(\psi_1)=\bigcap_{j=1}^N H_{ij}(\psi_1)$, where 
\begin{align*}
 H_{ij}(\psi):=\{x\in X\mid -c(x, y_i)-\psi^i\leq -c(x, y_j)-\psi^j\}.
\end{align*}
Since $\mu(\Lag_i(\psi_1)) = \weightvect^i > 0$ and $\rho$ is continuous, each $H_{ij}(\psi_1)$ has nonempty interior, hence by \eqref{Twist} combined with the implicit function theorem and the Lipschitzness of $\partial X$, we can see that each set $H_{ij}(\psi_1)$ has Lipschitz boundary. Since $\Lag_i(\psi_1)$ has nonempty interior, we see that it also has Lipschitz boundary.

In particular, this means $V \cap \interior{(\Lag_i(\psi_1))} \neq \emptyset$. Since $\rho > 0$ on $V \cap \interior{(\Lag_i(\psi_1))}$ which is open and non-empty, we have $\mu(V \cap \interior{(\Lag_i(\psi_1))}) > 0$ while $V \cap \interior{(\Lag_i(\psi_1))} \subset \Lag_i(\psi_1) \setminus \Lag_i(\psi_2)$. However this contradicts \cite[Remark 10.29]{Villani09}, as we must have $T_{\psi_1}=T_{\psi_2}$ $\mu$-a.e..
\end{proof}

\begin{lem}\label{inject}
	Suppose $\mu=\rho dx$ where $\mu$ satisfies a $(1, 1)$-PW inequality, and $\psi_1$, $\psi_2 \in \Psi_c$. Then   $\psi_1 - \psi_2 \in \spn(\onevect)$ if and only if $G(\psi_1) = G(\psi_2)$.
\end{lem}

\begin{proof}
	It is obvious from Definition \ref{def: lag cell} that $\psi_1 - \psi_2 \in \spn(\onevect)$ implies $G(\psi_1)=G(\psi_2)$, so we only show the opposite implication.
	
	Suppose $\weightvect:=G(\psi_1)=G(\psi_2)$ and let $\varphi_1 := \psi_1^{c^*}$, $\varphi_2 := \psi_2^{c^*}$. Also, write $T:=T_{\psi_1}=T_{\psi_2}$ (up to $\mu$-a.e.), which is the Monge solution to problem \eqref{eqn: monge} pushing $\mu$ forward to the discrete measure $\nu_\weightvect$. Finally, without loss of generality we may assume that $\weightvect^1 > 0$ and (by subtracting a multiple of $\onevect$) $\psi_1^1 = \psi_2^1$, and define $S := \{ i\in \{1, \ldots, N\}\mid  \psi_1^i = \psi_2^i \text{ and } \weightvect^i > 0 \}$. 

If we define the set
\begin{align*}
A := \bigcup_{i \in S} \Lag_i(\psi_1),
\end{align*}
then $\mu(A) \geq \weightvect^1 > 0$, and since it is a finite union of Laguerre cells, arguing as in the proof of Lemma  \ref{lem: lag cell lemma} we see $A$ has Lipschitz boundary. If $\mu(A) < 1$, since $\mu$ satisfies a $(1, 1)$-PW inequality, by \cite[Lemma 5.3]{KitagawaMerigotThibert19} we can conclude that $\int_{\partial A\cap \interior(X)}\rho d\H^{n-1}(x) > 0$. Then by \cite[(5.3)]{KitagawaMerigotThibert19}, we see  there exist $i \in S$, $j \not\in S$ and a point $x\in \Lag_i(\psi_1)\cap \Lag_j(\psi_1) \cap \partial A \cap \interior(X)$ where $\rho(x)>0$. Then $x \in \Lag_i(\psi_1)\cap \Lag_j(\psi_1) \subset \Lag_{i}(\psi_1)$ so by Lemma \ref{lem: lag cell lemma} above we must also have $x \in \Lag_{i}(\psi_2)$. Then we can calculate
\begin{align}\label{eqn: potential equals}
\varphi_1(x) + \psi_1^i = -c(x, y_i) = \varphi_2(x) + \psi_2^i \implies \varphi_1(x) = \varphi_2(x).
\end{align}
Arguing as in the proof of Lemma \ref{lem: lag cell lemma} above, since $x \in\Lag_i(\psi_1)\cap \interior(X)$ and $\rho(x) > 0$, we see that $\weightvect^j=\mu(\Lag_{j}(\psi_1)) > 0$. Since $x \in \Lag_i(\psi_1)\cap \Lag_j(\psi_1) \subset \Lag_{j}(\psi_1)$, we can apply Lemma \ref{lem: lag cell lemma} again to see $x \in \Lag_{j}(\psi_2)$. Hence
\begin{align*}
\varphi_1(x) + \psi_1^j = -c(x, y_j) = \varphi_2(x) + \psi_2^j \implies \psi_1^j = \psi_2^j,
\end{align*}
but this would imply $j \in S$, a contradiction. 

Now since $\mu(A) = 1$, the set $A\cap \rho^{-1}((0, \infty))$ must be dense in $\rho^{-1}((0, \infty))$. Then we can make the same calculation leading to \eqref{eqn: potential equals} above to find that $\varphi_1=\varphi_2$ on this dense set. Since $\phi_1$ and $\phi_2$ are $c^*$-transforms of vectors they are continuous on $\R^n$, thus they must actually be equal everywhere on $\rho^{-1}((0, \infty))$, hence on its closure $\spt \mu$.

With the above, we then see that
\begin{align*}
 \psi_1=\phi_1^{c^\dagger}=\phi_2^{c^\dagger}=\psi_2
\end{align*}
as desired.

\end{proof}

We are finally ready to prove the desired invertibility result.
\begin{prop}\label{homeomorphism}
Suppose $\mu=\rho dx$ satisfies a $(1, 1)$-PW inequality. Then $G: \conj{\mathcal{K}^0} / \onevect \to \weightvectset$ is a homeomorphism.
\end{prop}

\begin{proof}
	
	First let $f(\psi) = \psi^{c^*c^\dagger}-\psi$. Note that directly from Definition \ref{def: pseudo c-transforms}, for an arbitrary $x\in X$ we have  $\abs{\psi_1^{c^*}(x) - \psi_2^{c^*}(x)} \leq \norm{\psi_1 - \psi_2}_\infty$. 
	A similar calculation then yields
	\begin{align*}
	\norm{\psi_1^{c^*c^\dagger} - \psi_2^{c^*c^\dagger}}_\infty \leq \sup_{x\in \spt \mu}\abs{\psi_1^{c^*}(x) - \psi_2^{c^*}(x)} \leq \norm{\psi_1 - \psi_2}_\infty,
	\end{align*}
	hence by the triangle inequality, $f$ is continuous, in particular $\Psi_c=f^{-1}(\{0\})$ is closed. 
	
	Now for any $\psi\in \mathcal{K}^0$ it is clear there for each index $i$ must exist a point $x_i\in \spt\mu\cap \Lag_i(\psi)$, while just as in the proof of \cite[Proposition 4.1]{BansilKitagawa19a} we see that $\psi=\psi^{c^*c}$. Then for any $x\in X$, we would have
\begin{align*}
 -c(x_i, y_i)-\psi^{c^*}(x_i)=\psi^i=(\psi^{c^*c})^i\geq -c(x, y_i)-\psi^{c^*}(x),
\end{align*}
hence for such a $\psi$ we have
\begin{align*}
 \psi=\psi^{c^*c}=\psi^{c^*c^\dagger},
\end{align*}
in particular  ${\mathcal{K}^0} \subset \Psi_c$, thus $\conj {\mathcal{K}^0} \subset \Psi_c$. Then by Lemma \ref{inject}, $G(\psi_1) = G(\psi_2)$ if and only if $\psi_1 - \psi_2 \in \spn(\onevect)$ for $\psi_1$, $\psi_2 \in \conj{\mathcal{K}^0}$, and we obtain that the induced map (which we also call $G$) $G: \conj{\mathcal{K}^0} / \onevect \to \weightvectset$ is well-defined and injective. 
	
	Next note that $\conj{\mathcal{K}^0} / \onevect$ is closed and bounded and hence compact. %
	Hence, %
	$\weightvectset = \conj{G(\mathcal{K}^0)} \subset \conj{G(\conj{\mathcal{K}^0} / \onevect)} = G(\conj{\mathcal{K}^0} / \onevect)$.	Finally, since $G$ is a continuous bijection with compact domain it follows by \cite[Theorem 2.6.7]{Gamelin99} that $G$ is a homeomorphism. %
	
\end{proof}

\section{Quantitative Invertibility of $G$}\label{section: quantitative invertibility}
In this section we will add the assumption that $G$ is differentiable everywhere with continuous derivatives. This assumption is satisfied under the condition \eqref{QC}, but we note that we do not need the explicit geometric consequences of \eqref{QC} here, only the differentiability of $G$ for the results of this section.
\subsection{Alternative spectral estimates on $DG$}
We now obtain an estimate away from zero on the first nonzero eigenvalue of the mapping $DG$ over the set $\mathcal{K}^\thresh$ of a different nature than that of \cite[Theorem 5.1]{KitagawaMerigotThibert19}. The estimate there is of order $\thresh^3$ under the assumption of a $(1, 1)$-PW inequality, however we will show an estimate which is of order $N^{-4}\thresh^{\frac{1}{q}}$ under the assumption of a $(q, 1)$-PW inequality. As can be seen, in the case of $q=1$ we have traded two factors of $\thresh$ for factors of $N^{-2}$, this modification allows us to obtain quantitative estimates on the inverse of $G$, but as the parameter $\thresh\to 0$. In order to obtain a finite bound, we will be forced to use this new spectral estimate, along with taking $q>1$ in the Poincar{\'e}-Wirtinger inequality.

We start by showing the alternate estimate coming from assuming a $(q, 1)$-PW inequality on  $\mu$, versus a $(1, 1)$-PW inequality. We first recall some useful notation and definitions from \cite{KitagawaMerigotThibert19}.
\begin{defin}
	We will write $\interior(X)$ to denote the interior of the set $X$. Given an absolutely continuous measure $\mu=\rho dx$ and a set $A\subset X$ with Lipschitz boundary, we will write
	\begin{align*}
		\abs{\partial A}_{\rho}:&=\int_{\partial A\cap \interior(X)}\rho d\H^{n-1}(x),\quad
		\abs{A}_\rho:=\mu(A).
	\end{align*}
\end{defin}

\begin{lem}\label{lem: PW inequality bound}
	Suppose that $\mu=\rho dx$ satisfies a $(q,1)$-PW inequality where $q \geq 1$. Then
	\begin{align*}
	\inf_{A \subset X} \frac{\abs{\partial A}_{\rho}}{\min(\abs{A}_\rho, \abs{X \setminus A}_\rho)^{1/q}} \geq \frac{1}{2^{\frac{1}{q}}\Cpw},
	\end{align*}
	where the infimum is over $A\subset \interior(X)$ whose boundary is Lipschitz with finite $\mathcal{H}^{n-1}$-measure, and $\min(\abs{A}_\rho, \abs{X \setminus A}_\rho)>0$.
\end{lem}

\begin{proof}
	Let $A\subset \interior(X)$ be a Lipschitz domain as in the statement above, recall that we must have $q\leq \frac{n}{n-1}\leq 2$. Since we have a $(q, 1)$-PW inequality instead of a $(1, 1)$ inequality, by following the same method as \cite[Lemma 5.3]{KitagawaMerigotThibert19} we obtain the inequality	
	\begin{align*}
		\Cpw \abs{\partial A}_\rho &\geq \norm{\mathds{1}_A-\int_X \mathds{1}_A d\mu}_{L^q(\mu)}\\
		&=\(\int_A \abs{1-\abs{A}_\rho}^qd\mu+\int_{X\setminus A} \abs{\abs{A}_\rho}^qd\mu\)^{\frac{1}{q}} \\
		&=\(\abs{A}_\rho\abs{X\setminus A}_\rho^q+\abs{A}_\rho^q\abs{X\setminus A}_\rho\)^{\frac{1}{q}}\\
		&=\abs{A}_\rho^{\frac{1}{q}}\abs{X\setminus A}_\rho^{\frac{1}{q}}(\abs{X\setminus A}_\rho^{q-1}+\abs{A}_\rho^{q-1})^{\frac{1}{q}}\\
		&\geq \abs{A}_\rho^{\frac{1}{q}}\abs{X\setminus A}_\rho^{\frac{1}{q}}\\
		&\geq 2^{-\frac{1}{q}}\min(\abs{A}_\rho, \abs{X \setminus A}_\rho)^{1/q},
	\end{align*}
	hence taking an infimum gives the claim.
\end{proof}

Recall $DG$ is negative semidefinite on $\mathcal{K}^\thresh$ by \cite[Theorem 5.1]{KitagawaMerigotThibert19}. We work toward the following estimate.
\begin{thm}\label{BetterBound}
	Fix $\thresh>0$ and assume  $\mu=\rho dx$ satisfies a $(q,1)$-PW inequality where $q \geq 1$, then the second eigenvalue of $DG$ on $\mathcal{K}^\thresh$ is bounded above by $-\dfrac{2^{3-\frac{1}{q}}\thresh^{1/q}}{C_\nabla N^4\Cpw}<0$. 
	
\end{thm}

By \cite[(B1)]{KitagawaMerigotThibert19}, it can be seen that for almost every $\psi$ we have
\begin{align}\label{eqn: DG rep}
 D_iG^j(\psi)=D_jG^i(\psi)=\int_{\Lag_i(\psi)\cap \Lag_j(\psi)}\frac{\rho(x)}{\norm{\nabla_x c(x, y_i)-\nabla_xc(x, y_j)}}d\mathcal{H}^{n-1}(x).
\end{align} We now fix $\thresh>0$ and some $\psi\in \mathcal{K}^\thresh$ such that \eqref{eqn: DG rep} holds, and let $W$ be the (undirected) weighted graph constructed in \cite[Section 5.3]{KitagawaMerigotThibert19}: the vertices of $W$ consist of the collection $Y$, and for $i\neq j$ connect $y_i$ and $y_j$ by an edge of weight $w_{ij}$, defined by
\begin{align*}
 w_{ij}:=
 D_i G^j(\psi).%
\end{align*}

\begin{prop}\label{prop: connected subgraph}
	If $\mu=\rho dx$ satisfies a $(q,1)$-PW inequality where $q \geq 1$ and $\psi$ is such that \eqref{eqn: DG rep} holds, then $W$ is connected by edges of weight at least $\frac{2^{1-\frac{1}{q}}}{C_\nabla N^2\Cpw}\thresh^{1/q}$, that is: the weighted graph consisting of  all vertices of $W$ and only those edges of weight greater than or equal to $\frac{2^{1-\frac{1}{q}}}{C_\nabla N^2\Cpw}\thresh^{1/q}$ is connected. 
\end{prop}
\begin{proof}
Suppose by contradiction that the proposition is false. This implies that removing all edges with weight strictly less than $\frac{2^{1-\frac{1}{q}}}{C_\nabla N^2\Cpw}\thresh^{1/q}$ yields a disconnected graph. In other words, we can write $W = W_1 \cup W_2$ where $W_1$, $W_2 \neq \emptyset$ and are disjoint, such that every edge connecting a vertex in $W_1$ to a vertex in $W_2$ has weight strictly less than $\frac{2^{1-\frac{1}{q}}}{C_\nabla N^2\Cpw}\thresh^{1/q}$.  Letting $A := \cup_{y_i \in W_1} \Lag_i(\psi)$ we see that
	\begin{align*}
	\abs{\partial A}_{\rho} \leq 2C_\nabla\sum_{\{(i, j) \mid y_i\in W_1,\ y_j\in W_2\}} w_{ij} < \frac{2^{2-\frac{1}{q}}}{N^2\Cpw}\thresh^{1/q} \abs{W_1}\abs{W_2} \leq \frac{2^{2-\frac{1}{q}}}{N^2\Cpw}\thresh^{1/q} \frac{N^2}{4} = \frac{1}{2^{\frac{1}{q}}\Cpw}\thresh^{1/q}.
	\end{align*}
	On the other hand since both $W_1$ and $W_2$ are nonempty we have $\abs{A}_\rho, \abs{X \setminus A}_\rho \geq \thresh$. Hence
	\begin{align*}
	\frac{\abs{\partial A}_{\rho}}{\min(\abs{A}_\rho, \abs{X \setminus A}_\rho)^{1/q}} 
	< \frac{\thresh^{1/q}}{2^{\frac{1}{q}}\Cpw\thresh^{1/q}} = \frac{1}{2^{\frac{1}{q}}\Cpw}
	\end{align*}
	which contradicts Lemma \ref{lem: PW inequality bound}.
\end{proof}
Recall that given a weighted graph $W$, the \emph{weighted graph Laplacian} is the $N\times N$ matrix with entries
\begin{align*}
 L_{ij}:&=
\begin{cases}
 -w_{ij},&i\neq j,\\
 \sum_{k\in \{1, \ldots, N\}\setminus \{i\}} w_{ik},&i=j.
\end{cases}
\end{align*}
If $W$ is the graph we have defined above and $L$ its weighted graph Laplacian, then by \cite[Theorem 1.3]{KitagawaMerigotThibert19} we can see that $L=-DG(\psi)$.
\begin{proof}[Proof of Theorem \ref{BetterBound}]
First suppose $\psi$ satisfies \eqref{eqn: DG rep} and let $\ti W$ be the graph formed by dividing all of the edge weights in $W$ by $\frac{2^{1-\frac{1}{q}}\thresh^{1/q}}{C_\nabla N^2\Cpw}$. If $L$ and $\ti L$ are the weighted graph Laplacians of the graphs $W$ and $\ti W$ respectively, clearly $\ti L = \frac{C_\nabla N^2\Cpw}{2^{1-\frac{1}{q}}\thresh^{1/q}} L$.%

Now construct the graph $\widehat{W}$ from $\ti W$ by the following procedure: if an edge connecting $y_i$ and $y_j$ has weight $w_{ij}< 1$, we remove the edge, and if $w_{ij}\geq 1$, we set the weight of the edge equal to $1$. By Proposition \ref{prop: connected subgraph}, we see that $\widehat{W}$ is a connected graph whose edge weights are all $1$ over $N$ vertices, and in particular it has diameter $\diam(\widehat{W})=\sup\sum_{i, j}w_{ij}\leq N$, here the supremum is taken over all pairs of vertices in $\widehat{W}$ and collections of edges forming a path between those two vertices, and the sum runs over all edges in such a collection. Let us write $\widehat{L}$ for the graph Laplacian of $\widehat{W}$ and use $\lambda_2$ to denote the second eigenvalue of a positive semidefinite matrix. Then, using \cite[Lemma 3.2]{Fiedler75} to obtain the first inequality below and then \cite[Theorem 4.2]{Mohar91b} to obtain the second to final inequality, we find that
\begin{align*}
\lambda_2(-DG(\psi))&= \lambda_2(L)=\frac{2^{1-\frac{1}{q}}\thresh^{1/q}}{C_\nabla N^2\Cpw}\lambda_2(\ti L)\\
&\geq \frac{2^{1-\frac{1}{q}}\thresh^{1/q}}{C_\nabla N^2\Cpw}\lambda_2(\widehat{L})\geq \frac{2^{1-\frac{1}{q}}\thresh^{1/q}}{C_\nabla N^2\Cpw}\cdot \frac {4}{N \diam(\widehat{W})} \geq \frac {2^{3-\frac{1}{q}}\thresh^{1/q}} {C_\nabla N^4\Cpw}.
\end{align*}
Since  \eqref{eqn: DG rep} holds for almost every $\psi$, continuity of $DG$ finishes the proof.
\end{proof}
\subsection{Quantitative invertibility of $G$}

\begin{proof}[Proof of Theorem \ref{thm: quantitative invertibility}]
If $\min_i G^i(\psi_1)=\min_iG^i(\psi_2)=0$ there is nothing to prove, so assume  $\min_i G^i(\psi_1)>0$.

	By Proposition \ref{homeomorphism}, the restriction of $G$ to $\conj{\mathcal{K}^0} \cap \{ \psi \mid  \inner{\psi - \psi_1}{\onevect} = 0\}$ is invertible, let $H$ denote this inverse; by Theorem \ref{BetterBound} since $q\geq 1$ we see that
	\begin{align*}
	\norm{DH(\weightvect)} \leq \frac{C_\nabla N^4\Cpw}{4(\min_i \lambda^i)^{1/q}}. 
	\end{align*}
	We calculate, using Minkowski's integral inequality to obtain the first inequality,
	\begin{align*}
	\norm{\psi_1 - \psi_2}
	&= \(\sum_{i=1}^N {\(\int_{0}^{1} \inner{\nabla H^i(tG(\psi_1) + (1-t)G(\psi_2) )}{G(\psi_1) - G(\psi_2)}  dt \)}^2\)^{\frac{1}{2}} \\
	&\leq \int_{0}^{1}\(\sum_{i=1}^N { \inner{\nabla H^i(tG(\psi_1) + (1-t)G(\psi_2) )}{G(\psi_1) - G(\psi_2)}^2}\)^{\frac{1}{2}} dt \\
	&\leq  \norm{ G(\psi_1) - G(\psi_2) }   \int_{0}^{1} \norm{DH(tG(\psi_1) + (1-t)G(\psi_2) )} dt  \\
	&\leq \norm{ G(\psi_1) - G(\psi_2) }   \int_{0}^{1} \frac{C_\nabla N^4\Cpw}{4(\min_i (tG^i(\psi_1) + (1-t)G^i(\psi_2)))^{1/q}} dt  \\
	&\leq \norm{ G(\psi_1) - G(\psi_2) }   \int_{0}^{1} \frac{C_\nabla N^4\Cpw}{4(t\min_i G^i(\psi_1))^{1/q}} dt  \\
	&= \frac{N^4C_\nabla \Cpw q}{4(q-1)}\frac{\norm{ G(\psi_1) - G(\psi_2) }}{\min_i G^i(\psi_1)^{1/q}},
	\end{align*}
	here it is crucial that $q>1$ in order to obtain the final line.  If $\min_i G^i(\psi_1)=0$ we may switch the roles of $\psi_1$ and $\psi_2$, which yields the claimed bound.
\end{proof}

\section{Stability in Hausdorff Distance}\label{section: hausdorff estimates}

We will now work towards proving Corollary \ref{cor: hausdorff convergence}, our quantitative stability of Laguerre cells measured in the Hausdorff distance. In Theorem \ref{thm: HausPsiBound} below, we obtain quantitative control of the Hausdorff distance between different Laguerre cells in terms of the dual vectors. However, we would like to obtain the bound in terms of data that is readily available, i.e. the masses of the respective target measures, and our quantitative invertibility result Theorem \ref{thm: quantitative invertibility} will allow us to write the bound in these terms. Starting in this section, we also assume $X$ is  $c$-convex with respect to $Y$ (so in particular, $X$ has Lipschitz boundary) and $c$ satisfies \eqref{QC}. In contrast to the previous section, we will need the full geometric power of \eqref{QC}. We also write $\mathcal{H}^k$ for the $k$-dimensional Hausdorff measure.

\begin{rmk}\label{rmk: optimality}
 The goal of this section will be Theorem \ref{thm: HausPsiBound}, which effectively shows the map $(\R^N, \norm{\cdot}_{\infty})\ni\psi\mapsto \Lag_{i}(\psi)\in (\{\text{convex, compact sets}\}, d_\H)$ is locally $\frac{1}{n}$-H\"older. This estimate is likely \emph{not} sharp in the H\"older exponent, and for the canonical case $c(x, y)=-\inner{x}{y}$, the map can be shown to be locally Lipschitz. We present a quick proof here which depends on the explicit form of the Laguerre cells for this special choice of the cost function, based on an idea suggested by the anonymous referee. As the proof is specific to the inner product cost, it is not clear how to obtain this improvement in the more general case of cost satisfying \eqref{QC}.
 
 Indeed, let $R>0$ sufficiently large so that $\spt\mu\subset B_R(0)\subset \R^n$ and $c(x, y)=-\inner{x}{y}$ and assume $\psi_1$, $\psi_2\in \R^N$ are such that $\Lag_i(\psi_1)\cap \Lag_i(\psi_2)$ has nonempty interior for some $i$. Recall that for any set $S$, its \emph{support function} is defined for any $v\in \R^n$ by
\begin{align*}
 h_S(v):=\sup_{x\in S}\inner{v}{x}.
\end{align*}
It is known that (see \cite[Section 1.7]{Schneider93}) for two convex sets $S_1$ and $S_2$, $h_{S_1+S_2}=h_{S_1}+h_{S_2}$ (where $S_1+S_2$ is the Minkowski sum), and $h_{S_1\cap S_2}$ is the closure of the function
\begin{align*}
 v\mapsto \inf\{h_{S_1}(v_1)+h_{S_2}(v_2)\mid v_1+v_2=v\},
\end{align*}
while, 
\begin{align}\label{eqn: Hausdorff dist support functions}
 d_\H(S_1, S_2)=\sup_{\S^{n-1}}\abs{h_{S_1}-h_{S_2}}.
\end{align}
Now let $x_0$ be in the interior of $\Lag_i(\psi_1)\cap \Lag_i(\psi_2)$, then we can write
\begin{align*}
\Lag_i(\psi_1)+\{-x_0\}=\bigcap_{j\neq i}\{\inner{\cdot}{y_i-y_j}\leq \tilde\psi_1^j\}\cap B_R(x_0), 
\end{align*}
where each $\tilde\psi_1^j:=\psi^i_1-\psi^j_1-\inner{x_0}{y_i-y_j}>0$ with a uniform lower bound $\delta>0$. Moreover, this bound can be estimated using the maximal radius of a ball centered at $x_0$ remaining in $\Lag_i(\psi_1)\cap \Lag_i(\psi_2)$; by using Lemma \ref{lem: contains ball} and the calculation leading to \eqref{eqn: lower bound intersection}, one finds $\delta$ stays uniformly away from zero if $\norm{\psi_1-\psi_2}_\infty$ is sufficiently small. A similar statement holds for $\Lag_i(\psi_2)+\{-x_0\}$. We can then calculate, for any $v\in \S^{n-1}$,
\begin{align*}
 h_{\Lag_i(\psi_1)+\{-x_0\}}(v)&=\inf\{\sum_{j\neq i}t_j\tilde\psi^j_1+\abs{u}R-\inner{u}{x_0}\mid t_j\geq 0, u\in \R^n, u+\sum_{j\neq i}t_j (y_i-y_j)=v\}.
\end{align*}
By taking all $t_j=0$ and $u=v$, we see $h_{\Lag_i(\psi_1)-x_0}(v)\leq R-\inner{v}{x_0}\leq 2R$, while since $x_0\in B_R(0)$, for any choices of $t_j\geq 0$, 
\begin{align*}
 h_{\Lag_i(\psi_1)+\{-x_0\}}(v)\geq \delta \sum_{j\neq i}t_j.
\end{align*}
This shows that in the infimum in the expression for $h_{\Lag_i(\psi_1)+\{-x_0\}}(v)$, the $t_j$ can be taken to satisfy the further restriction $\sum_{j\neq i}t_j\leq \frac{2R}{\delta}$, and the same statement holds for $h_{\Lag_i(\psi_2)+\{-x_0\}}(v)$. Now suppose $\{s_j\}$ and $u_0$ achieve the infimum in $h_{\Lag_i(\psi_2)+\{-x_0\}}(v)$, then we obtain
\begin{align*}
 h_{\Lag_i(\psi_1)}(v)-h_{\Lag_i(\psi_2)}(v)&=h_{\Lag_i(\psi_1)(v)+\{-x_0\}}(v)-h_{\Lag_i(\psi_2)+\{-x_0\}}(v)\\
 &\leq \sum_{j\neq i}s_j\tilde\psi^j_1+\abs{u_0}R-\inner{u_0}{x_0}-(\sum_{j\neq i}s_j\tilde\psi^j_2+\abs{u_0}R-\inner{u_0}{x_0})\\
 &= \sum_{j\neq i}s_j(\psi^i_1-\psi^i_2-\psi^j_1+\psi^j_2)\leq \frac{4R}{\delta}\norm{\psi_1-\psi_2}_\infty.
\end{align*}
A symmetric calculation reversing the roles of $\psi_1$ and $\psi_2$, and then taking a supremum over $v\in \S^{n-1}$ combined with \eqref{eqn: Hausdorff dist support functions} shows the claimed local Lipschitz bound. 
\end{rmk}
\begin{defin}
	We denote $\omega_j = \frac{\pi^{j/2}}{\Gamma(\frac {j} {2} + 1)}$
	for the volume of the unit ball in $\R^j$. 
\end{defin}
We start with a simple lemma in convex geometry.
\begin{lem}\label{lem: contains ball}
	If $A$ is a bounded convex set with $\L(A) > 0$ then $A$ contains a ball of radius $\CA{\L(A)}$ where 
	\begin{align*}
	\CA := \frac{2^{n-1}}{\omega_{n}(n+2)^n\diam(A)^{n-1}}.%
	\end{align*}
\end{lem}

\begin{proof}
Let $\simp$ be a simplex in $A$ with volume at least $\frac{1}{(n+2)^n} \L(A)$ as given by the main theorem of \cite{Lassak11}. Since $\simp$ is convex and is contained in a ball of radius $\frac{\diam(A)}{2}$, we have $\mathcal{H}^{n-1}(\partial \simp)\leq n\omega_{n}\(\frac{\diam(A)}{2}\)^{n-1}$ (see \cite[p. 211]{Schneider93}). Then it is standard that $\simp$ contains a ball of radius $r$, where
	\begin{align*}
	r = \frac{n\vol(\simp)}{\mathcal{H}^{n-1}(\partial \simp)} \geq \frac{2^{n-1}\L(A)}{\omega_{n}(n+2)^n\diam(A)^{n-1}},
	\end{align*}
see for example the last formula in the proof of \cite[Corollary 3]{VaughanHyman67} and the discussion following it. 
\end{proof}
In the next proposition, we estimate the term $\sup_{x \in B} d(x, A)$ from the definition of Hausdorff distance by the Lebesgue measure of the difference of the two sets, when they are convex. We opt to take a different approach from the proof of Theorem \ref{thm: symmetric convergence}: ultimately we will control the Lebesgue measure of the symmetric difference of Laguerre cells directly by the dual variables $\psi$, then attempt to quantitatively invert the map $G$, allowing us to invoke the first estimate in Theorem \ref{thm: symmetric convergence}.
\begin{prop}\label{setdiffbound}
	
	Let $A \subset B$ be bounded convex sets. Then
	\begin{align*}
	\L(B \setminus A) \geq \frac{\omega_{n}{(\sup_{x \in B} d(x, A))^n}}{(2\pi)^{n-1}} \(\arccos(1 - \frac{2\CA^2\L(A)^2}{\diam(B)^2}) \)^{n-1}.
	\end{align*}	

\end{prop}

\begin{proof}
	
	If $\L(A)=0$ the claim is clear, thus assume $\L(A)>0$. Let $D_A = 2\CA{\L(A)}$ be the diameter of the ball contained in $A$ from Lemma \ref{lem: contains ball}. 
	
	Let $x \in B \setminus A $ be arbitrary. We shall first consider the case where $n=2$. \begin{figure}[H] 
  \centering
    \includegraphics[width=0.5\textwidth]{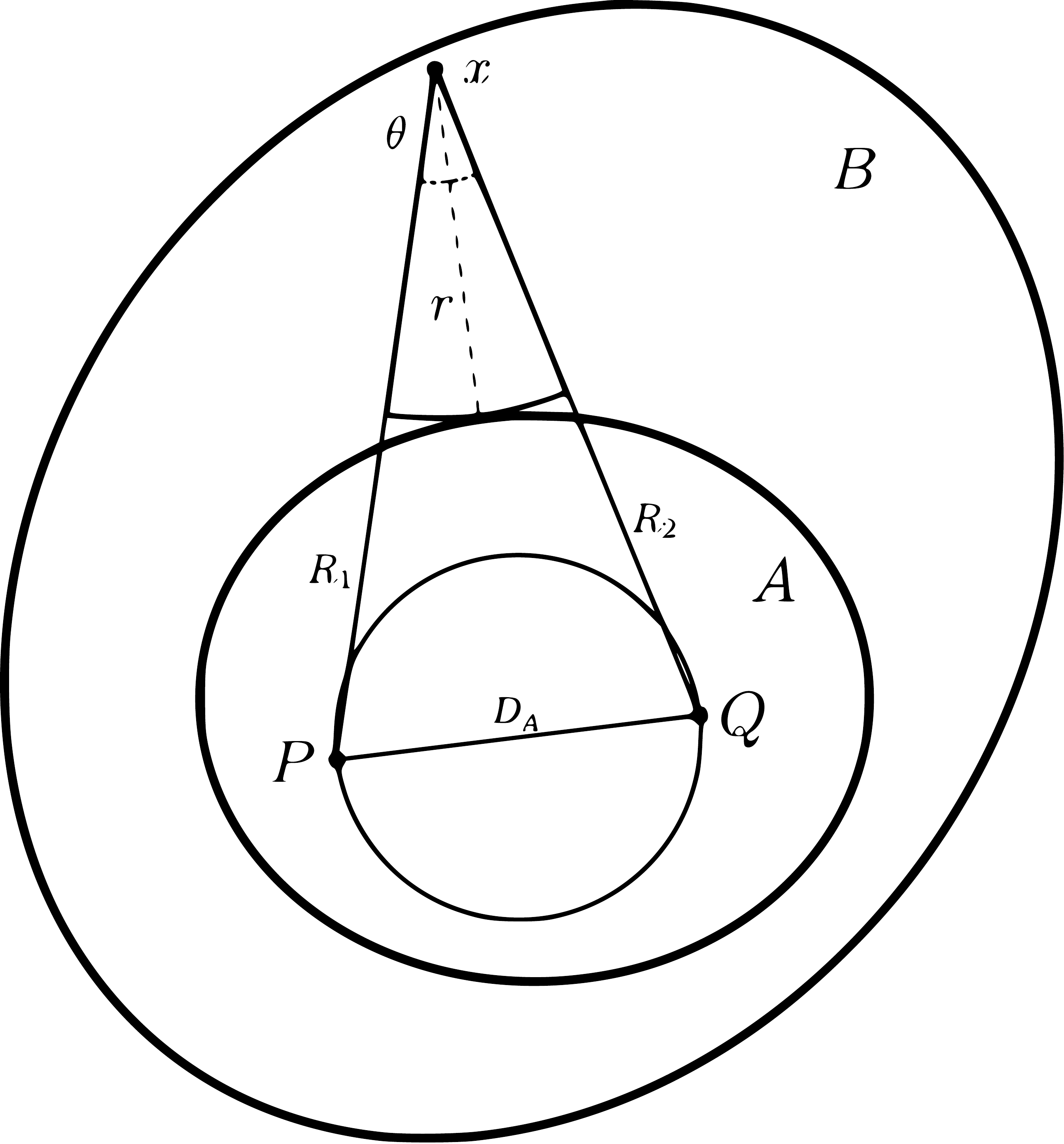}
     \caption{}\label{figure: fig1}
\end{figure}
	
	First $P$, $Q$ are points chosen on the boundary of the disk contained in $A$ so that $R_1 = R_2$ where $R_1$ and $R_2$ are the lengths of the segments $xP$ and $xQ$ (such $P$, $Q$ exist by a continuity argument, see Figure \ref{figure: fig1}). Set $r := d(x, A)$. Next let $S$ be the shaded circular sector, i.e. $S := B_r(x) \cap \Delta(P, Q, x)$ where $\Delta(P, Q, x)$ is the triangle with vertices $P$, $Q$, $x$. Let $\theta$ be the measure of the angle $\angle PxQ$ and set $R := R_1 = R_2$. 
	
	Note that $S \subset B \setminus A$. Then by the law of cosines
	\begin{align*}
	2R^2 - 2R^2\cos \theta &=R_1^2 + R_2^2 - 2R_1 R_2\cos \theta = D_A^2 \\
	\implies \cos \theta &=   1 - \frac{D_A^2}{2R^2} \leq 1 - \frac{D_A^2}{2\diam(B)^2}.
	\end{align*}
Thus we estimate the area of $S$ as
	\begin{align*}
	\pi r^2 \frac{\theta}{2\pi} 
	\geq \frac{r^2}{2} \arccos(1 - \frac{D_A^2}{2\diam(B)^2}) 
	= \frac 12 d(x, A)^2 \arccos(1 - \frac{D_A^2}{2\diam(B)^2}).
	\end{align*}
Since $x\in B$ was arbitrary we obtain
	\begin{align*}
	\L(B \setminus A) \geq \frac 12 \sup_{x \in B} d(x, A)^2 \arccos(1 - \frac{D_A^2}{2\diam(B)^2})
	\end{align*}
	as desired.
	
	Now in higher dimensions the construction above yields a spherical sector instead of the circular sector, $S$. By slicing with planes through $x$ and the center of the ball and applying the argument used when $n=2$ we see that this  spherical sector has angle $\theta$ in all directions. Hence we calculate that the volume of our spherical sector is estimated as
	\begin{align*}
	\omega_{n} r^n (\frac{\theta}{2\pi})^{n-1}
	\geq \frac{\omega_{n}{r^n}}{(2\pi)^{n-1}} \(\arccos(1 - \frac{D_A^2}{2\diam(B)^2}) \)^{n-1}
	= \frac{\omega_{n}{d(x, A)^n}}{(2\pi)^{n-1}} \(\arccos(1 - \frac{D_A^2}{2\diam(B)^2}) \)^{n-1}.
	\end{align*}
	Hence we have
	\begin{align*}
	\L(B \setminus A) \geq \frac{\omega_{n}{\sup_{x \in B} d(x, A)^n}}{(2\pi)^{n-1}} \(\arccos(1 - \frac{D_A^2}{2\diam(B)^2}) \)^{n-1}
	\end{align*}
	as desired. 
\end{proof}
The following lemma is a simple use of the coarea formula to control the Lebesgue measure of the difference of Laguerre cells corresponding to different dual variables $\psi_1$ and $\psi_2$ in terms of the difference $\norm{\psi_1-\psi_2}_\infty$, similar to the proof of \cite[Proposition 4.8]{KitagawaMerigotThibert19}. For any index $i\in \{1, \ldots, N\}$ and a set $E\subset \R^n$, we will use the notation
\begin{align*}
 \coord{E}{i}:=\invcExp{i}{E}.
\end{align*}
\begin{lem}\label{celldiffbound}
	Let $\psi_1, \psi_2 \in \R^n$. Then for some universal $\Cdiff>0$,
	\begin{align*}
	\L({\Lag_{i}(\psi_1)} \setminus {\Lag_{i}(\psi_2)}) \leq %
	\Cdiff N\norm{\psi_1 - \psi_2}_\infty.
	\end{align*}
\end{lem}

\begin{proof}
	Suppose $x \in {\Lag_{i}(\psi_1)} \setminus {\Lag_{i}(\psi_2)}$, then there is a $k \neq i$ so that $c(x, y_k) + \psi_2^k < c(x, y_i) + \psi_2^i$ while $c(x, y_i) + \psi_1^i \leq c(x, y_k) + \psi_1^k$, combining these yields
	\begin{align*}
	\psi_2^k - \psi_2^i < c(x, y_i) - c(x, y_k) \leq \psi_1^k - \psi_1^i.
	\end{align*}	
	Hence writing $f_k(x) = c(x, y_i) - c(x, y_k)$,
	\begin{align}\label{eqn: cell containment}
	{\Lag_{i}(\psi_1)} \setminus {\Lag_{i}(\psi_2)} \subset \bigcup_{k\neq i} f_k\i ([\psi_2^k - \psi_2^i, \psi_1^k - \psi_1^i]).
	\end{align}

	We proceed to bound $\L(f_k\i ([\psi_2^k - \psi_2^i, \psi_1^i - \psi_1^k]))$ using the coarea formula. We have
	\begin{align*}
	\L(f_k\i([a,b]))
	&= \int_{f_k\i([a,b])} d\L(x) = \int_{a}^b \int_{f\i(\{t\})} \frac{1}{\norm{\nabla f_k(x)}}  d{\mathcal{H}}^{n-1}(x) dt \\
	&\leq \frac{b-a}{\eps_{\mathrm{tw}}} (\sup_{t \in [a,b]} {\mathcal{H}}^{n-1}\(f_k\i(\{t\})\))
	\end{align*}
where we recall $\eps_{\mathrm{tw}}$ is from Definition \ref{def: universal}.
	
	Next we bound $\sup_{t \in (a,b)} {\mathcal{H}}^{n-1}(f_k\i(\{t\}))$. Let $A^k_t := \{x\in X\mid f_k(x) \leq t \}$. We claim that $f_k\i(\{t\}) \subset \partial A^k_t$. Clearly $f_k\i(\{t\}) \subset \conj{A^k_t}$. Suppose by contradiction there is $x \in f_k\i(\{t\}) \cap \interior{A^k_t}$. Then $x$ has an open neighborhood $U$ so that for every $y \in U$, $f_k(y) \leq t = f_k(x)$. In particular $f_k(x)$ is a local maximum and so $\nabla f_k(x) = 0$, contradicting \eqref{Twist}. 
	
	By \eqref{QC}, $[A^k_t]_i$ is convex and contained in $[X]_i$. Hence $\mathcal{H}^{n-1}([\partial A^k_t]_i) =  \mathcal{H}^{n-1} (\partial [A^k_t]_i) \leq \mathcal{H}^{n-1}(\partial [X]_i) = \mathcal{H}^{n-1}([\partial X]_i)$ (again see \cite[p. 211]{Schneider93}). Then 
	we have ${\mathcal{H}}^{n-1}(f_k\i(\{t\})) \leq \mathcal{H}^{n-1} (\partial [A^k_t]_i)\leq C_{\exp}^{n-1} \mathcal{H}^{n-1}(\partial X)$, and combining with above
	\begin{align*}
	\L(f_k\i([a,b]))
	\leq \frac{b-a}{\eps_{\mathrm{tw}}} (\sup_{t \in [a,b]} {\mathcal{H}}^{n-1}(f_k\i(\{t\})))
	\leq \frac{C_{\exp}^{n-1}\mathcal{H}^{n-1}(\partial X)}{\eps_{\mathrm{tw}}} (b-a).
	\end{align*}
	Since $\psi_1^k - \psi_1^i - (\psi_2^k - \psi_2^i) \leq 2 \norm{\psi_1 - \psi_2}_{\infty}$, by combining the above with \eqref{eqn: cell containment} we have
	\begin{align*}
	\L({\Lag_{i}(\psi_1)} \setminus {\Lag_{i}(\psi_2)}) \leq 
	\sum_{k \neq i} \L(f_k\i ([\psi_2^k - \psi_2^i, \psi_1^k - \psi_1^i]))
	\leq \frac{2C_{\exp}^{n-1}N\mathcal{H}^{n-1}(\partial X)}{\eps_{tw}} \norm{\psi_1 - \psi_2}_{\infty}
	\end{align*}
	as desired. 
\end{proof}

Finally, we apply the bound in Proposition \ref{setdiffbound} to the images of Laguerre cells under the coordinates induced by the maps $\invcExp{i}{\cdot}$, which are convex by \eqref{QC}. Combining with Lemma \ref{celldiffbound} above allows us to control the Hausdorff distance between Laguerre cells by the difference of the dual variables defining the cells.
\begin{thm}\label{thm: HausPsiBound}
	Suppose that 
	\begin{align}\label{eqn: bound on difference}
	 \norm{\psi_1 - \psi_2}_\infty < \frac{\max(\L(\Lag_{i}(\psi_1)), \L({\Lag_{i}(\psi_2)}))}{2\Cdiff N}
	 \end{align}
	  where $\Cdiff$ is the constant from Lemma \ref{celldiffbound}. Then for some universal constants $C_1>0$ and $C_2>0$,
	\begin{align*}
	d_\H(\Lag_{i}(\psi_1), \Lag_{i}(\psi_2) )^n \leq \frac{C_1N\norm{\psi_1 - \psi_2}_\infty}{\left(\arccos(1-C_2\max(\L(\Lag_{i}(\psi_1)), \L(\Lag_{i}(\psi_2)))^2)\right)^{n-1}}.
	\end{align*}
\end{thm}

\begin{proof}
By \eqref{QC}, we see that $\coord{\Lag_i(\psi)}{i}$ is a convex set for any $i$.

Applying Proposition \ref{setdiffbound} with $A = {\coord{\Lag_{i}(\psi_1)}{i}} \cap {\coord{\Lag_{i}(\psi_2)}{i}}$ and $B = {\coord{\Lag_{i}(\psi_1)}{i}}$ we obtain
	\begin{align*}
	&\L({\coord{\Lag_{i}(\psi_1)}{i}} \setminus {\coord{\Lag_{i}(\psi_2)}{i}}) \\
	&\geq \frac {\omega_n (\sup_{x\in \Lag_{i}(\psi_1)} (d(\invcExp{i}{x}, A)))^n}{(2\pi)^{n-1}} \(\arccos(1 - \frac{2\CA^2\L(A)^2}{\diam(\coord{\Lag_{i}(\psi_1)}{i})^2})\)^{n-1}\\
	&\geq \frac {\omega_n (\sup_{x\in \Lag_{i}(\psi_1)} (d(\invcExp{i}{x}, \coord{\Lag_{i}(\psi_2)}{i})))^n}{(2\pi)^{n-1}} \(\arccos(1 - \frac{2\CA^2\L(A)^2}{\diam(\coord{\Lag_{i}(\psi_1)}{i})^2})\)^{n-1}
	\end{align*}
	as $\coord{\Lag_{i}(\psi_1)}{i} \setminus ({\coord{\Lag_{i}(\psi_1)}{i}} \cap {\coord{\Lag_{i}(\psi_2)}{i}}) = \coord{\Lag_{i}(\psi_1)}{i} \setminus {\coord{\Lag_{i}(\psi_2)}{i}}$. Similarly, we also see
	\begin{align*}
	&\L({\coord{\Lag_{i}(\psi_2)}{i}} \setminus {\coord{\Lag_{i}(\psi_1)}{i}}) \\
	&\geq \frac {\omega_n (\sup_{x\in \Lag_{i}(\psi_2)} (d(\invcExp{i}{x}, {\coord{\Lag_{i}(\psi_1)}{i}} )))^n}{(2\pi)^{n-1}} \(\arccos(1 - \frac{2\CA[A]^2\L(A)^2}{\diam(\coord{\Lag_{i}(\psi_2)}{i})^2})\)^{n-1}
	\end{align*}
	and so
	\begin{align}
	&\max(\L({\coord{\Lag_{i}(\psi_2)}{i}} \setminus {\coord{\Lag_{i}(\psi_1)}{i}}), \L({\coord{\Lag_{i}(\psi_1)}{i}} \setminus {\coord{\Lag_{i}(\psi_2)}{i}})) \notag\\
	&\geq 	\frac {\omega_n d_\H({\coord{\Lag_{i}(\psi_1)}{i}}, {\coord{\Lag_{i}(\psi_2)}{i}})^n}{(2\pi)^{n-1}} \min_{j=1, 2}\(\(\arccos(1 - \frac{2\CA[A]^2\L(A)^2}{\diam(\coord{\Lag_{i}(\psi_j)}{i})^2})\)^{n-1}\).\label{eqn: max lag diff bound}
	\end{align}

Suppose $\L(\Lag_{i}(\psi_1))\geq \L({\Lag_{i}(\psi_2)})$ (the other case can be handled with a symmetric argument). Then using Lemma \ref{celldiffbound} and the assumption \eqref{eqn: bound on difference} on $\norm{\psi_1-\psi_2}_\infty$, for both $j=1$ or $2$,
\begin{align}
 \frac{2\CA[A]^2\L(A)^2}{\diam(\coord{\Lag_{i}(\psi_j)}{i})^2}&= \frac{2^{2n-1}\L(A)^2}{\omega_n^2(n+2)^{2n}\diam(A)^{2n-2}\diam(\coord{\Lag_{i}(\psi_j)}{i})^2}\notag\\
&\geq \frac{2^{2n-1}(\L({\coord{\Lag_{i}(\psi_1)}{i}})-\L({\coord{\Lag_{i}(\psi_1)}{i}} \setminus {\coord{\Lag_{i}(\psi_2)}{i}}))^2}{\omega_n^2(n+2)^{2n}\diam(X)^{2n}}\notag\\
&\geq \frac{2^{2n-1}\L({\coord{\Lag_{i}(\psi_1)}{i}})^2}{4\omega_n^2(n+2)^{2n}\diam(X)^{2n}}\notag\\
&\geq\frac{2^{2n-1}C_{\exp}^{2n}\max(\L(\Lag_{i}(\psi_1)), \L(\Lag_{i}(\psi_2)))^2}{4\omega_n^2(n+2)^{2n}\diam(X)^{2n}}.\label{eqn: lower bound intersection}
\end{align}
Combining the above estimate with Lemma \ref{celldiffbound} and \eqref{eqn: max lag diff bound},
	\begin{align*}
	CN\norm{\psi_1 - \psi_2}_\infty &\geq \max(\L({\coord{\Lag_{i}(\psi_2)}{i}} \setminus {\coord{\Lag_{i}(\psi_1)}{i}}), \L({\coord{\Lag_{i}(\psi_1)}{i}} \setminus {\coord{\Lag_{i}(\psi_2)}{i}}))\\
	&\geq \frac {\omega_n d_\H({\coord{\Lag_{i}(\psi_1)}{i}}, {\coord{\Lag_{i}(\psi_2)}{i}})^n}{(2\pi)^{n-1}} \(\arccos(1-C_2\max(\L(\Lag_{i}(\psi_1)), \L(\Lag_{i}(\psi_2)))^2\)^{n-1}.
	\end{align*}

Since the map $\invcExp{i}{\cdot}$ is bi-Lipschitz with universal Lipschitz constants, there is some universal $C>0$ such that
	\begin{align*}
	Cd_\H(\Lag_{i}(\psi_1), \Lag_{i}(\psi_2))^n &\leq d_\H({\coord{\Lag_{i}(\psi_1)}{i}}, {\coord{\Lag_{i}(\psi_2)}{i}})^n,
	\end{align*}
	finishing the proof.
\end{proof}

With these preliminary results in hand, are finally ready to prove Corollary \ref{cor: hausdorff convergence}.
\begin{proof}[Proof of Corollary \ref{cor: hausdorff convergence}]
	To obtain statement (1), since  $\norm{\weightvect_k-\weightvect_0}\to 0$ as $k\to\infty$, by Proposition  \ref{homeomorphism} we must have $\psi_k \to \psi$. Combining this with Theorem \ref{thm: HausPsiBound} gives (1).
	
To show claim (2), assume $q>1$. Combining \eqref{eqn: constraint} and Theorem \ref{thm: quantitative invertibility}, for the choice of $C_1=\frac{qC_\Delta C_\nabla \Cpw\norm{\rho}_{C^0(X)}}{2(q-1)}$ we have
\begin{align*}
 &\norm{\psi_1 - \psi_2}_\infty\leq\norm{\psi_1 - \psi_2} \leq \frac{ qN^4C_\nabla\Cpw \norm{ \weightvect_1-\weightvect_2 }}{4(q-1)\max(\min_i \weightvect_1^i, \min_i\weightvect_2^i)^{1/q} }\\
 &< \frac{\max(\weightvect_1^i, \weightvect_2^i)}{2\Cdiff N\norm{\rho}_{C^0(X)}}= \frac{\max(\mu(\Lag_i(\psi_1)), \mu(\Lag_i(\psi_2)))}{2\Cdiff N\norm{\rho}_{C^0(X)}}
\leq  \frac{\max(\L(\Lag_i(\psi_1)), \L(\Lag_i(\psi_2)))}{2\Cdiff N}.
\end{align*}
Hence we can apply Theorem \ref{thm: HausPsiBound} and Theorem \ref{thm: quantitative invertibility} to obtain
\begin{align*}
	&d_\H({\Lag_{i}(\psi_1)}, {\Lag_{i}(\psi_2)} )^n \leq \frac{C_1N\norm{\psi_1 - \psi_2}_\infty}{\left(\arccos(1-C_2\max(\L(\Lag_{i}(\psi_1)), \L(\Lag_{i}(\psi_2)))^2)\right)^{n-1}}\\
	&\leq \frac{ qC_1N^5C_\nabla\Cpw \norm{ \weightvect_1-\weightvect_2 }}{4(q-1)\max(\min_i \weightvect_1^i, \min_i\weightvect_2^i)^{1/q}\left(\arccos(1-C_2\max(\L(\Lag_{i}(\psi_1)), \L(\Lag_{i}(\psi_2)))^2)\right)^{n-1}}\\
	&\leq \frac{ qC_1N^5C_\nabla\Cpw \norm{ \weightvect_1-\weightvect_2 }}{4(q-1)\max(\min_i \weightvect_1^i, \min_i\weightvect_2^i)^{1/q}\left(\arccos(1-C_2\norm{\rho}_{C^0(X)}^{-2}\max(\weightvect_1^i, \weightvect_2^i)^2)\right)^{n-1}}
	\end{align*}
	where we have used that $t\mapsto \frac{1}{\arccos(1-t)}$ is a decreasing function and $\L(\Lag_i(\psi_j))\geq \norm{\rho}_{C^0(X)}^{-1}\weightvect_j^i$.
\end{proof}

\section{Quantitative uniform convergence of dual potentials}\label{sec: uniform convergence}
In this final section, we prove Theorem \ref{thm: hausdorff and uniform convergence}, showing that the uniform difference of dual potentials can be controlled by the Hausdorff distance between Laguerre cells. In this section, we assume all of the conditions of the previous section, except that $\mu$ satisfies a $(q, 1)$-PW inequality. We comment that if $\mu$ is assumed to satisfy a $(q,1)$-PW inequality with $q>1$, we may applying the quantitative invertibility result Theorem \ref{thm: quantitative invertibility} to obtain the bound on the uniform difference in terms of the difference of the masses of the target measures.

We start with a basic lemma.
\begin{lem}\label{lem: leb symmetric difference hausdorff bound}
	
	If $A, B \subset X$ are bounded convex sets then $\L(A \Delta B) \leq 2d_\H(A,B)\H^{n-1}(\partial X)$. 
	
\end{lem}
\begin{proof}
 Denote by $A_\eps$ the closed $\eps$ neighborhood of $A$. Then using the first displayed equation on p. 221 in \cite[III.13.3]{Santalo04} combined with the fact that if $A\subset B$ with $A$ convex, then $\mathcal{H}^{n-1}(\partial A)\leq \mathcal{H}^{n-1}(\partial B)$, we obtain 
 \begin{align*}
\L(A_\eps) \leq \L(A) + \eps \H^{n-1}(\partial A_\eps).
\end{align*}
Then noting that $B \subset A_{d_\H(A,B)}$ and vice versa, we obtain the claim.
\end{proof}

\begin{prop}\label{prop: hausdorff and uniform convergence quant}
Suppose $\psi_1$, $\psi_2\in \R^N$ with $\inner{\psi_1-\psi_2}{\onevect}=0$ and $\Lag_i(\psi_1)$, $\Lag_i(\psi_2)\neq \emptyset$ for each $i\in \{1, \ldots, N\}$. Then 
\begin{align}\label{eqn: potential bound}
\norm{\psi_1 - \psi_2} \leq \frac{ N^4C_\nabla\Cpw n\H^{n-1}(\partial X)\sqrt{\sum_{i=1}^N d_\H(\Lag_i(\psi_1),\Lag_i(\psi_2))^2}}{2(\max(\min_i \L(\Lag_i(\psi_1)), \min_i(\L(\Lag_i(\psi_2)))))^{1-\frac{1}{n}} \L(X)^{\frac{1}{n}}}.
\end{align}	
\end{prop}

\begin{proof}
	
Define $\ti \mu := \frac{1}{\L(X)} \L \bigg|_X$. Note that since $X$ is connected $\ti \mu$ satisfies an $( \frac{n}{n-1},1)$-PW inequality. Next define $\wv_i = \ti \mu (\Lag(\psi_i))$ for $i=1$, $2$. We see that for any $i$,
\begin{align*}
\abs{\wv_1^i - \wv_2^i} 
&= \abs{\ti \mu(\Lag_i(\psi_1)) - \ti \mu(\Lag_i(\psi_2))} \\
&=\frac{1}{\L(X)} \abs{\L(\Lag_i(\psi_1) \setminus \Lag_i(\psi_2)) - \L(\Lag_i(\psi_2) \setminus \Lag_i(\psi_1))} \\
&\leq \frac{1}{\L(X)} \abs{\L(\Lag_i(\psi_1) \setminus \Lag_i(\psi_2))} + \abs{ \L(\Lag_i(\psi_2) \setminus \Lag_i(\psi_1))} \\
&= \frac{\L(\Lag_i(\psi_1)\Delta\Lag_i(\psi_2))}{\L(X)} \leq \frac{2\H^{n-1}(\partial X)}{\L(X)}d_\H(\Lag_i(\psi_1),\Lag_i(\psi_2)),
\end{align*}
where we have used Lemma \ref{lem: leb symmetric difference hausdorff bound} to obtain the last inequality above. Hence
\begin{align*}
\norm{\wv_1 - \wv_2} 
\leq \frac{2\H^{n-1}(\partial X)}{\L(X)} \sqrt{\sum_{i=1}^N d_\H(\Lag_i(\psi_1),\Lag_i(\psi_2))^2}.
\end{align*}
Then we can apply Theorem \ref{thm: quantitative invertibility} using $\ti\mu$ in place of $\mu$ to obtain \eqref{eqn: potential bound} as desired.

\end{proof}
\begin{proof}[Proof of Theorem \ref{thm: hausdorff and uniform convergence}]
For any $\psi_1$, $\psi_2\in \R^N$, by definition of $c^*$-transform we have $\norm{\psi_1^{c^*}-\psi_2^{c^*}}_{C^0(X)}\leq \norm{\psi_1-\psi_2}_\infty\leq \norm{\psi_1-\psi_2}$. Thus the theorem follows from Proposition \ref{prop: hausdorff and uniform convergence quant} above.
\end{proof}

\bibliographystyle{alpha}
\bibliography{bansil-kitagawa-stability}

\end{document}